\documentclass[a4paper]{amsart}                

\usepackage[latin1]{inputenc}                  
\usepackage[T1]{fontenc}                       
\usepackage[spanish,english]{babel}            
\usepackage[pdftex]{graphicx}                  
\usepackage{amsmath,amssymb,amsthm}            
\usepackage{latexsym}                          
\usepackage{delarray}                          
\usepackage{bbm}                               
\usepackage{hyperref}                          
\usepackage[pdftex,usenames,dvipsnames]{color} 
\usepackage{bbding,trfsigns}                   
\usepackage{wasysym}                           
\usepackage{datetime}                          

\DeclareMathAlphabet{\mathpzc}{OT1}{pzc}{m}{it} 
\newcommand{\B}{\mathbb{B}}
\newcommand{\C}{\mathbb{C}}

\newcommand{\E}{\mathbb{E}}
\newcommand{\EE}{\mathcal{E}}
\newcommand{\G}{\Gamma}
\newcommand{\HH}{\mathbb{H}}

\newcommand{\N}{\mathbb{N}}
\newcommand{\PP}{\mathbb{P}}
\newcommand{\Q}{\mathbb{Q}}
\newcommand{\QQ}{\mathcal{Q}}
\newcommand{\R}{\mathbb{R}}
\newcommand{\Rn}{{\mathbb{R}^n}}

\newcommand{\W}{\mathbb{W}}

\DeclareMathOperator{\supp}{supp}

\newtheorem{ThA}{Theorem}

\newtheorem{Th}{Theorem}[section]              

\newtheorem{Prop}{Proposition}[section]
\newtheorem{Lem}{Lemma}[section]

\title[Square functions and spectral multipliers for Bessel operators in UMD spaces]
      {Square functions and spectral multipliers for Bessel operators in UMD spaces}

\author[J.J. Betancor]{Jorge J. Betancor}
\author[A.J. Castro]{Alejandro J. Castro}
\author[L. Rodr\'{\i}guez-Mesa]{L. Rodr\'{\i}guez-Mesa}

\address{\newline
        Jorge J. Betancor, Alejandro J. Castro and Lourdes Rodr\'iguez-Mesa \newline
        Departamento de An\'alisis Matem\'atico,
        Universidad de La Laguna, \newline
        Campus de Anchieta, Avda. Astrof\'{\i}sico Francisco S\'anchez, s/n, \newline
        38271, La Laguna (Sta. Cruz de Tenerife), Spain}
\email{jbetanco@ull.es, ajcastro@ull.es, lrguez@ull.es}



\thanks{The authors are partially supported by MTM2010/17974. The second author is also supported by a FPU grant from the Government of Spain.}

\begin{document}

\footnotetext{Date: \today.}

\maketitle                                  

\begin{abstract}
    In this paper we consider square functions (also called Littlewood-Paley $g$-functions) associated to Hankel
    convolutions acting on functions in the Bochner Lebesgue space $L^p((0,\infty ),\mathbb{B})$, where $\mathbb{B}$
    is a UMD Banach space. As special cases we study square functions defined by fractional derivatives of the Poisson
    semigroup for the Bessel operator $\Delta _\lambda =-x^{-\lambda }\frac{d}{dx}x^{2\lambda }\frac{d}{dx}x^{-\lambda }$, $\lambda >0$.
    We characterize the UMD property for a Banach space $\mathbb{B}$ by using $L^p((0,\infty ),\mathbb{B})$-boundedness
    properties of $g$-functions defined by Bessel-Poisson semigroups. As a by product we prove that the fact that the imaginary
    power $\Delta_\lambda^{i\omega}$, $\omega \in \mathbb{R}\setminus\{0\}$, of the Bessel operator $\Delta _\lambda $ is bounded in
    $L^p((0,\infty ),\mathbb{B})$, $1<p<\infty$, characterizes the UMD property for the Banach space $\mathbb{B}$.
    As applications of our results for square functions we establish the boundedness in $L^p((0,\infty ),\mathbb{B})$
    of spectral multipliers $m(\Delta _\lambda )$ of Bessel operators defined by functions $m$ which are holomorphic in sectors $\Sigma_\vartheta$.
\end{abstract}

\section{Introduction}\label{sec:intro}
Square functions (also called Littlewood-Paley $g$-functions) were considered in the works of Littlewood, Paley,
Zygmund and Marcinkiewicz  during the decade of the 30's in the last century. These functions were introduced to
get new equivalent norms, for instance, in $L^p$-spaces. By using these new equivalent norms the boundedness of
some operators, for instance multipliers, can be established.

Suppose that $(\Omega, \Sigma, \mu)$ is a  $\sigma$-finite measure space and $\{T_t\}_{t>0}$ is a symmetric
diffusion semigroup of operators in the sense of Stein (\cite{Ste1}). For every $k\in \mathbb{N}$, the square
function $g_k$ associated to $\{T_t\}_{t>0}$ is defined by
$$g_k(\{T_t\}_{t>0})(f)(x)
    =\left(\int_0^\infty \left| t^{k}\partial_t^k T_t(f)(x)\right|^2 \frac{dt}{t}\right)^{1/2}.$$

In \cite[p. 120]{Ste1} it is shown that, for every $k\in \mathbb{N}$ and $1<p<\infty$, there exists $C>0$ such that
\begin{equation}\label{equiv}
    \frac{1}{C}\|f-E_0(f)\|_{L^p(\Omega ,\mu )}
        \leq \| g_k(\{T_t\}_{t>0})(f) \|_{L^p(\Omega ,\mu )}
        \leq C\|f\|_{L^p(\Omega ,\mu )},\quad f\in L^p(\Omega ,\mu ),
\end{equation}
where $E_0(f)=\lim_{t\rightarrow \infty}T_t(f)$ is the projection onto the fixed point space of $\{T_t\}_{t>0}$.
As application of \eqref{equiv} it can be proved that Laplace transform type multipliers associated with $\{T_t\}_{t>0}$
are bounded from $L^p(\Omega ,\mu )$ into itself, $1<p<\infty$ (\cite[p. 121]{Ste1}). Note that when   $E_0=0,$ \eqref{equiv}
says that by defining for every $1<p<\infty$ and $k\in \mathbb{N}$,
$$||| f |||_k
    =\| g_k(\{T_t\}_{t>0})(f) \|_{L^p(\Omega ,\mu )},\quad f\in L^p(\Omega ,\mu ),$$
$|||\cdot |||_k$ is an equivalent norm  to the usual norm in $L^p(\Omega,\mu )$. Meda in \cite{Me1} extends the property
\eqref{equiv} to symmetric contraction semigroups $\{T_t\}_{t>0}$ with $E_0=0$ and he applies it to get the boundedness in
$L^p$ of spectral multipliers $m(\mathcal{L})$ where the operator $\mathcal{L}$ is the infinitesimal generator of $\{T_t\}_{t>0}$
and $m$ is an holomorphic and bounded function in a sector $\Sigma_\vartheta =\{z\in \mathbb{C}: |\mbox{Arg }z|<\vartheta \}$. Here $1<p<\infty$
and $\vartheta \in [0,\pi /2]$ are connected.

We consider the functions
\begin{equation}\label{2.1}
    \W_t(z)
        =\frac{e^{-|z|^2/4t}}{(4\pi t)^{n/2}},\quad z\in \mathbb{R}^n\mbox{ and }t>0,
\end{equation}
and
$$\PP_t(z)
    =b_n\frac{t}{(t^2+|z|^2)^{(n+1)/2}},\quad z\in \mathbb{R}^n\mbox{ and }t>0,$$
where $b_n=\pi ^{-(n+1)/2}\Gamma ((n+1)/2)$.

As it is well-known the classical heat semigroup $\{\mathbb{W}_t\}_{t>0}$ is defined by
$$\W_t(f)(x)
    =\int_{\mathbb{R}^n} \W_t(x-y)f(y)dy,\quad  x\in \mathbb{R}^n \mbox{ and } t>0,$$
and the classical Poisson semigroup $\{\PP_t\}_{t>0}$ is given by
$$\PP_t(f)(x)
    =\int_{\mathbb{R}^n} \PP_t(x-y)f(y)dy,\quad x\in \mathbb{R}^n \mbox{ and } t>0 ,$$
for every $f\in L^p(\mathbb{R}^n)$. $\{\W_t\}_{t>0}$ and
$\{\PP_t\}_{t>0}$ are generated by $-\Delta$ and $-\sqrt{\Delta}$,
respectively, where $\Delta =-\sum_{j=1}^n \frac{\partial
^2}{\partial x_i^2}$ denotes the Laplace operator. The classical
heat and Poisson semigroups are the first examples of diffusion
semigroups having trivial fixed point space. For every measurable
function $g:\mathbb{R}^n\longrightarrow \mathbb{C}$, we define
$g_t(x)=t^{-n}g(x/t)$, $x\in \mathbb{R}^n$ and $t>0$.

Let $k\in \mathbb{N}$. We can write
$$g_k(\{\W_t\}_{t>0})(f)(x)
    = \left(\int_0^\infty |(\varphi _{\sqrt{t}}*f)(x)|^2\frac{dt}{t}\right)^{1/2},\quad x\in \mathbb{R}^n,$$
being $\varphi (x)=\partial _t^kG_{\sqrt{t}}(x)_{|t=1}$ and $G(x)=(4\pi )^{-n/2}e^{-|x|^2/4}$, $x\in \mathbb{R}^n$. Also, we have that
$$g_k(\{\PP_t\}_{t>0})(f)(x)
    =\left(\int_0^\infty |(\phi _t*f)(x)|^2\frac{dt}{t}\right)^{1/2},\quad x\in \mathbb{R}^n,$$
where $\phi (x)=\partial _t^k \PP_{t}(x)_{|t=1}$, $x\in \mathbb{R}^n$.

If $\psi \in L^2(\mathbb{R}^n)$ we consider the square function defined by
$$g_\psi (f)(x)
    =\left(\int_0^\infty  |(\psi _t*f)(x)|^2\frac{dt}{t}\right)^{1/2},\quad x\in \mathbb{R}^n.$$
Thus, $g_\psi$ includes as special cases $g_k(\{\W_t\}_{t>0})$ and $g_k(\{\PP_t\}_{t>0})$.

In the sequel, if $f\in \mathcal{S}(\mathbb{R}^n)$, the Schwartz class, we denote by $\widehat{f}$ the Fourier transform of $f$ given by
$$\widehat{f}(y)
    =\frac{1}{(2\pi )^{n/2}}\int_{\mathbb{R}^n}f(x)e^{-ix\cdot y}dx,\quad y\in \mathbb{R}^n.$$
As it is well-known the Fourier transform can be extended to $L^2(\mathbb{R}^n)$ as an isometry in $L^2(\mathbb{R}^n)$.

\begin{ThA}\label{ThA}
    Suppose that $\psi \in L^2(\mathbb{R}^n)$ satisfies the following properties:

    (i) if $\alpha =(\alpha _1,...,\alpha_n)\in \{0,1\}^n$ and $|\alpha |=\sum_{j=1}^n\alpha _j\leq 1+[n/2]$,
    the distributional derivative $\frac{\partial ^{|\alpha |}}{\partial x_1^{\alpha _1}...\partial x_n^{\alpha _n}}\widehat{\psi}$
    is represented by a measurable function and
    $$\sup_{|z|=1}\int_0^\infty t^{2|\alpha |}\left|\left(\frac{\partial ^{|\alpha |}}{\partial x_1^{\alpha _1}...\partial x_n^{\alpha _n}}\widehat{\psi}\right)(tz)\right|^2\frac{dt}{t}<\infty,$$

    (ii) $\displaystyle \inf_{|z|=1}\int_0^\infty |\widehat{\psi }(tz)|^2\frac{dt}{t}>0$.

    Then, for every $1<p<\infty$, there exists $C>0$ such that
    $$\frac{1}{C}\|f\|_{L^p(\mathbb{R}^n)}
        \leq \|g_\psi (f)\|_{L^p(\mathbb{R}^n)}\leq C\|f\|_{L^p(\mathbb{R}^n)},\quad f\in L^p(\mathbb{R}^n).$$
\end{ThA}

Also, square functions can be defined by using functional calculus for operators
(see, for instance, \cite{LeMer2} and \cite{Me2}). Note that if $A$ is the infinitesimal generator of the semigroup
$\{T_t\}_{t>0}$ we can write, for every $k\in \mathbb{N}$,
$$t^k\partial _t^kT_t
    =F_k(tA),\quad t>0,$$
where $F_k(z)=(-z)^ke^{-z}$, $z\in \mathbb{C}$.

Suppose that $\mathbb{B}$ is a Banach space and $T$ is a linear bounded operator from
$L^p(\Omega ,\mu )$ into itself where $1<p<\infty$. We define $T\otimes I_\mathbb{B}$ on $L^p(\Omega ,\mu )\otimes \mathbb{B}$
in the usual way. If $T$ is positive, $T\otimes I_\mathbb{B}$ can be extended to $L^p(\Omega ,\mu ,\mathbb{B})$
as a bounded operator from $L^p(\Omega ,\mu ,\mathbb{B})$ into itself, and to simplify we continue denoting this extension by $T$.

The question is to give a definition for the square functions, when we consider functions taking values in a Banach space $\mathbb{B}$,
defining equivalent norms in the Bochner-Lebesgue space $L^p(\Omega, \mu ,\mathbb{B})$, $1<p<\infty$.

Let $\{T_t\}_{t>0}$ be a symmetric diffusion semigroup on a $\sigma$-finite measure space $(\Omega, \Sigma ,\mu )$.
We denote by $\{P_t\}_{t>0}$ the subordinated semigroup to $\{T_t\}_{t>0}$, that is,
$$P_t(f)
    =\frac{t}{2\sqrt{\pi}}\int_0^\infty \frac{e^{-t^2/(4u)}}{u^{3/2}}T_u(f)du,\quad t>0.$$
Thus $\{P_t\}_{t>0}$ is also a symmetric diffusion semigroup. The classical Poisson semigroup is the subordinated semigroup of the classical heat semigroup.

In order to define $g$-functions in a Banach valued setting, the more natural way is to replace the absolute value
in the scalar definition by the norm in $\mathbb{B}$. This is the way followed, for instance, in \cite{MTX} and \cite{Xu}
where they work with square functions defined by subordinated semigroups $\{P_t\}_{t>0}$ as follows
$$g_{1,\mathbb{B}}(\{P_t\}_{t>0})(f)(x)
    =\left(\int_0^\infty \|t\partial_tP_t(f)(x)\|_{\mathbb{B}}^2\frac{dt}{t}\right)^{1/2},\quad x\in \Omega .$$
Actually in \cite{MTX} and \cite{Xu} generalized square functions are considered where the $L^2$-norm is replaced by the
$L^q$-norm, $1<q<\infty$. As a consequence of \cite[Theorems 5.2 and 5.3]{MTX} (see also \cite{Kw}) we deduce that
for a certain $1<p<\infty$ there exists $C>0$ such that, for every $f\in L^p(\Rn,\mathbb{B})$,
$$ \frac{1}{C} \|f\|_{L^p(\Rn,\mathbb{B})}
    \leq \|g_{1,\mathbb{B}}(\{\PP_t\}_{t>0})(f) \|_{L^p(\Rn )}
    \leq C \|f\|_{L^p(\Rn,\mathbb{B})},$$
if, and only if, $\mathbb{B}$ is isomorphic to a Hilbert space.

In order to get new equivalent norms for $L^p(\Omega, \mu ,\mathbb{B})$ by using square functions,
for Banach spaces $\mathbb{B}$ which are not isomorphic to Hilbert spaces, stochastic integrals and $\gamma$-radonifying
ope\-rators have been considered. We point out the papers of Bourgain \cite{Bou}, Hyt\"onen \cite{Hy}, Hyt\"onen, Van Neerven and
Portal \cite{HNP}, Hyt\"onen and Weis \cite{HW},  Kaiser \cite{Ka},  and  Kaiser and Weis \cite{KaWe}, amongst others.
In this work we use $\gamma$-radonifying operators. We recall some definitions and properties about this kind of operators that will be useful in the sequel.

Assume that $\mathbb{B}$ is a Banach space and $H$ is a Hilbert space. We choose a sequence $\{\gamma _j\}_{j\in \mathbb{N}}$
of independent standard Gaussian variables defined on some probability space $(\Omega, \mathcal{F},\rho)$.
By $\mathbb{E}$ we denote the expectation with respect to $\rho$. A linear operator $T:H\longrightarrow \mathbb{B}$ is said to be $\gamma$-summing ($T\in \gamma ^\infty (H,\mathbb{B})$) when
$$\|T\|_{\gamma ^\infty (H,\mathbb{B})}
    =\sup\left(\mathbb{E} \Big\|\sum_{j=1}^k\gamma _jT(h_j)\Big\|_\mathbb{B}^2\right)^{1/2}<\infty ,$$
where the supremum is taken over all the finite family $\{h_j\}_{j=1}^k$ of orthonormal vectors in $H$.
$\gamma ^\infty (H,\mathbb{B})$ endowed with the norm $\|\cdot\|_{\gamma ^\infty (H,\mathbb{B})}$ is a Banach space.
We say that a linear operator $T:H\longrightarrow \mathbb{B}$ is $\gamma$-radonifying (shortly, $T \in \gamma(H,\B)$) when
$T\in \overline{\mathcal{F}(H,\mathbb{B})}^{\gamma ^\infty (H,\mathbb{B})}$, where $\mathcal{F}(H,\mathbb{B})$ denotes
the span of finite range operators from $H$ to $\B$. If $\mathbb{B}$ does not contain isomorphic copies of $c_0$ then,
$\gamma (H,\mathbb{B})=\gamma ^\infty (H,\mathbb{B})$ (\cite{HJ}, \cite{Kw1} and \cite[Theorem 5.9]{Nee}). Note that if $\mathbb{B}$ is UMD,
$\mathbb{B}$ does not contain isomorphic copies of $c_0$. If $H$ is separable and $\{h_n\}_{n\in \mathbb{N}}$ is an
orthonormal basis of $H$, a linear operator $T:H\longrightarrow \mathbb{B}$ is $\gamma$-radonifying, if and only if,
the series $\sum_{j=1}^\infty \gamma _jT(h_j)$ converges in $L^2(\Omega ,\mathbb{B})$ and, in this case,
$$\|T\|_{\gamma ^\infty (H,\mathbb{B})}
    =\left(\mathbb{E}\Big\|\sum_{j=1}^\infty \gamma _jT(h_j)\Big\|_\mathbb{B}^2\right)^{1/2}.$$
In the sequel we write $\|\cdot\|_{\gamma (H,\mathbb{B})}$ to refer to $\|\cdot\|_{\gamma ^\infty (H,\mathbb{B})}$ when acts on $\gamma (H,\mathbb{B})$.

Throughout this paper we always consider $H=L^2((0,\infty ),dt/t)$.
Suppose that $f:(0,\infty )\longrightarrow \mathbb{B}$ is a measurable function such that $S\circ f\in H$,
for every $S\in \mathbb{B}^*$, the dual space of $\mathbb{B}$. Then, there exists a bounded linear operator
$T_f:H\longrightarrow \mathbb{B}$ (shortly, $T_f\in L(H,\mathbb{B})$) such that
$$\langle S,T_f(h)\rangle_{\B^*,\B}
    =\int_0^\infty \langle S,f(t)\rangle_{\B^*,\B} h(t)\frac{dt}{t},\quad h\in H\mbox{ and }S\in \mathbb{B}^*.$$
We say that $f\in \gamma ((0,\infty ),dt/t,\mathbb{B})$ provided that $T_f\in \gamma (H,\mathbb{B})$.
If $\mathbb{B}$ does not contain isomorphic copies of $c_0$, then the space $\{T_f: f\in \gamma ((0,\infty ),dt/t,\mathbb{B})\}$
is dense in $\gamma (H,\mathbb{B})$. It is usual to identify $f$ and $T_f$.

Banach spaces with the UMD property play an important role in our results.
The Hilbert transform $\mathcal{H}(f)$ of $f\in L^p (\mathbb{R})$, $1\leq p<\infty$, is defined by
$$\mathcal{H}(f)(x)
    =\lim_{\varepsilon \rightarrow 0^+} \frac{1}{\pi} \int_{|x-y|>\varepsilon }\frac{f(y)}{x-y}dy,\quad \mbox{ a.e.  }x\in\mathbb{R}.$$
As it is well-known the Hilbert transform defines a bounded linear operator from $L^p(\mathbb{R})$ into itself, $1<p<\infty$,
and from $L^1(\mathbb{R})$ into $L^{1,\infty }(\mathbb{R})$. $\mathcal{H}$ is defined on $L^p(\mathbb{R})\otimes \mathbb{B}$, $1\leq p<\infty$,
in a natural way. We say that $\mathbb{B}$ is a UMD space when the Hilbert transform can be extended to $L^p(\mathbb{R},\mathbb{B})$
as a bounded operator from $L^p(\mathbb{R},\mathbb{B})$ into itself for some (equivalently, for every) $1<p<\infty$.
There exist a lot of characterizations of UMD Banach spaces. The papers of  Bourgain \cite{Bou} and Burkholder \cite{Bu3} have
been fundamental in the development of the theory of Banach spaces with the UMD property. UMD Banach spaces are the suitable setting
to analyze vector valued Littlewood-Paley functions.

Kaiser and Weis \cite{KaWe} (see also \cite{Ka}) considered, for every $\psi \in L^2(\mathbb{R}^n)$,
the operator (usually called wavelet transform associated to $\psi$) $\mathcal{W}_\psi$ defined by
$$\mathcal{W}_\psi (f)(t,x)
    =(\psi_t *f)(x),\quad x\in \mathbb{R}^n \text{ and } t>0,$$
for every $f\in \mathcal{S}(\mathbb{R}^n,\mathbb{B})$, the $\B$-valued Schwartz space.

The following result was established in \cite[Theorem 4.2]{KaWe}.
\begin{ThA} \label{ThB}
    Suppose that $\mathbb{B}$ is a UMD Banach space with Fourier type $r\in (1,2]$ and $\psi \in L^2(\mathbb{R}^n)$ satisfies the following two conditions:

    (i) If $\alpha =(\alpha_1,...,\alpha _n)\in \{0,1\}^n$ and $|\alpha |\leq 1+[n/r]$, the distributional
    derivative $\frac{\partial ^{|\alpha |}}{\partial x_1^{\alpha _1}...\partial x_n^{\alpha _n}}\widehat{\psi}$ is represented by a measurable function and
    $$\sup_{|z|=1}\int_0^\infty t^{2|\alpha |}\left|\left(\frac{\partial ^{|\alpha |}}{\partial x_1^{\alpha _1}...\partial x_n^{\alpha _n}}\widehat{\psi}\right)(tz)\right|^2\frac{dt}{t}<\infty.$$

    (ii) $\displaystyle \inf_{|z|=1}\int_0^\infty |\widehat{\psi }(tz)|^2\frac{dt}{t}>0$.

    \noindent  Then, for every $1<p<\infty$ there exists $C>0$ such that
    $$\frac{1}{C}\|f\|_{L^p(\mathbb{R}^n,\mathbb{B})}
        \leq \|\mathcal{W}_\psi (f)\|_{L^p(\mathbb{R}^n,\gamma (H,\mathbb{B}))}\leq C\|f\|_{L^p(\mathbb{R}^n,\mathbb{B})},\quad f\in \mathcal{S}(\mathbb{R}^n,\mathbb{B}).$$
\end{ThA}

Note that since $\gamma  (H,\mathbb{C})=H$, Theorem \ref{ThB} can be seen as a vector-valued generalization of Theorem \ref{ThA}.
We recall that every UMD Banach space has Fourier type greater than 1 (\cite{Bou1}) and the complex plane $\mathbb{C}$ has Fourier type 2.

Our objective in this paper is to get new equivalent norms for $L^p((0,\infty ),\mathbb{B})$, when $\mathbb{B}$
is a UMD Banach space, by using square functions involving Hankel convolutions and Poisson semigroups associated
with Bessel operators. These square functions allow us to obtain new characterizations of UMD Banach spaces.
We also describe the UMD property by the boundedness in $L^p((0,\infty ),\mathbb{B})$, $1<p<\infty$, of
the imaginary power $\Delta _\lambda^{i\omega}$, $\omega \in \mathbb{R}\setminus\{0\}$, of the Bessel operator
$\Delta _\lambda =-x^{-\lambda }\frac{d}{dx}x^{2\lambda }\frac{d}{dx}x^{-\lambda}$, on $(0,\infty )$.
As a consequence of our results about square functions in the Bessel setting we obtain $L^p((0,\infty ),\mathbb{B})$-boundedness
properties for spectral multipliers associated with Bessel operators.

If $J_\nu$ denotes the Bessel function of the first kind and order $\nu >-1$, we have that
\begin{equation}\label{BesselJ}
    \Delta_{\lambda ,x}(\sqrt{xy}J_{\lambda -1/2}(xy))=y^2\sqrt{xy}J_{\lambda -1/2}(xy),\quad x,y\in (0,\infty ).
\end{equation}
Here and in the sequel, unless otherwise stated, we assume that $\lambda >0$. The Hankel transform $h_\lambda (f)$ of $f\in L^1(0,\infty )$ is defined by
$$h_\lambda (f)(x)
    =\int_0^\infty \sqrt{xy}J_{\lambda -1/2}(xy)f(y)dy,\quad x\in (0,\infty ).$$
This transform plays in the Bessel setting the same role as the Fourier transformation in the classical (Laplacian) setting (see \eqref{BesselJ}).

We consider the space $\mathcal{S}_\lambda (0,\infty )$ of all those smooth functions
$\phi$ on $(0,\infty )$ such that, for every $m,k \in \mathbb{N}$,
$$\eta _{m,k}^\lambda (\phi )
    =\sup_{x\in (0,\infty )} x^m\left|\left(\frac{1}{x}\frac{d}{dx}\right)^k(x^{-\lambda }\phi (x))\right|<\infty .$$
If $\mathcal{S}_\lambda (0,\infty )$ is endowed with the topology generated by the family
$\{\eta _{m,k}^\lambda \}_{m,k\in \mathbb{N}}$ of seminorms, $\mathcal{S}_\lambda (0,\infty )$ is a
Fr\'echet space and $h_\lambda $ is an isomorphism on $\mathcal{S}_\lambda (0,\infty )$ (\cite[Lemma 8]{Ze}).
Moreover, $h_\lambda ^{-1}=h_\lambda$ on $\mathcal{S}_\lambda (0,\infty)$.
The Hankel transformation $h_\lambda$ can be also extended to $L^2(0,\infty )$ as an isometry (\cite[p. 214 and Theorem 129]{Ti}).

By adapting the results in \cite{Hi}, we define the Hankel convolution $f\# _\lambda g$ of $f,g\in L^1((0,\infty ), x^\lambda dx)$ by
$$(f\# _\lambda g)(x)
    =\int_0^\infty f(y) \;_\lambda \tau _x(g)(y)dy,\quad x\in (0,\infty ),$$
where the Hankel translation $\;_\lambda \tau _x(g)$ of $g$ is given by
$$\;_\lambda \tau _x(g)(y)
    =\frac{(xy)^\lambda }{\sqrt{\pi}2^{\lambda -1/2}\Gamma (\lambda )}\int_0^\pi (\sin \theta )^{2\lambda -1}\frac{g(\sqrt{(x-y)^2+2xy(1-\cos \theta )})}{((x-y)^2+2xy(1-\cos \theta ))^{\lambda /2}}d\theta ,\quad x,y\in (0,\infty ).$$
Note that there is not a group operation $\circ$ on $(0,\infty )$ for which
$$\;_\lambda \tau _x(g)(y)
    =g(x\circ y),\quad x,y\in (0,\infty ).$$
The following interchange formula holds
\begin{equation}\label{6.1}
    h_\lambda (f\# _\lambda g)
        =x^{-\lambda }h_\lambda (f)h_\lambda (g),\quad f,g\in L^1((0,\infty ), x^\lambda dx).
\end{equation}
If $\psi$ is a measurable function on $(0,\infty )$ we define
$$\psi _{(t)}(x)
    =\psi _{(t)}^\lambda (x)=\frac{1}{t^{\lambda +1}}\psi \Big(\frac{x}{t}\Big),\quad t,x\in (0,\infty ).$$
If $\psi \in \mathcal{S}_\lambda (0,\infty )$ and $\mathbb{B}$ is a Banach space, we define the operator (Hankel wavelet transform)
$\mathcal{W}_{\psi ,\mathbb{B}}^\lambda $ as follows:
$$\mathcal{W}_{\psi ,\mathbb{B}}^\lambda (f)(t,x)
    =(\psi _{(t)} \# _\lambda f)(x),\quad t,x\in (0,\infty ),$$
for every $f\in L^p((0,\infty ),\mathbb{B})$, $1<p<\infty$.

We establish in our first result a Hankel version of Theorem \ref{ThB}.

\begin{Th}\label{ThBHankel}
    Let $\mathbb{B}$ be a UMD Banach space, $\lambda >0$ and $1<p<\infty$. Suppose that $\psi \in \mathcal{S}_\lambda (0,\infty )$
    is not identically zero and $\int_0^\infty x^\lambda \psi(x) dx=0$. Then, there exists $C>0$ such that
    $$\frac{1}{C}\|f\|_{L^p((0,\infty ),\mathbb{B})}
        \leq \|\mathcal{W}_{\psi ,\mathbb{B}}^\lambda (f)\|_{L^p((0,\infty ),\gamma (H,\mathbb{B}))}
        \leq C \|f\|_{L^p((0,\infty ),\mathbb{B})},$$
    for every $f\in L^p((0,\infty ),\mathbb{B})$.
\end{Th}

Harmonic analysis associated with Bessel operators was firstly analyzed by Muckenhoupt and Stein \cite{MS}.
Recently, that study has been completed (see \cite{BBFMT}, \cite{BCR1} and \cite{BFMT}).

The Poisson semigroup $\{P_t^\lambda \}_{t>0}$ associated to the operator $\Delta _\lambda $ is defined as follows:
$$P_t^\lambda (f)(x)
    =\int_0^\infty P_t^\lambda (x,y)f(y)dy,\quad t,x\in (0,\infty ),$$
for every $f\in L^p(0,\infty )$, $1\leq p<\infty$. The Poisson kernel $P_t^\lambda (x,y)$, $t,x,y\in (0,\infty )$, is defined by (see \cite{Wei})
$$P_t^\lambda (x,y)
    =\frac{2\lambda (xy)^ \lambda t}{\pi }\int_0^\pi \frac{(\sin \theta )^{2\lambda -1}}{((x-y)^2+t^2+2xy(1-\cos \theta ))^{\lambda +1}}d\theta ,\quad t,x,y\in (0,\infty ).$$
For every $t>0$, we can write
$$P_t^\lambda (f)
    =K_{(t)}^\lambda \# _\lambda f,\quad f\in L^p(0,\infty ), \ 1\leq p<\infty ,$$
where
$$K^\lambda (x)
    =\frac{2^{\lambda +1/2}\Gamma (\lambda +1)}{\sqrt{\pi }}\frac{x^\lambda }{(1+x^2)^{\lambda +1}},\quad x\in (0,\infty ).$$
Then, for every $k\in \mathbb{N}$ and $f\in L^p(0,\infty )$, $1<p<\infty $,
$$g_k(\{P_t^\lambda \}_ {t>0})(f)(x)
    =\left(\int_0^\infty |(t^k\partial ^k_tK_{(t)}^\lambda \# _\lambda f)(x)|^2\frac{dt}{t}\right)^{1/2},\quad x\in (0,\infty ).$$
$g$-functions in the Bessel setting were studied in \cite{BFS}.

In \cite{BFMT} it was considered the square function defined by
$$g_{1,\mathbb{B}}(\{P_t^\lambda \}_ {t>0})(f)(x)
    =\left(\int_0^\infty \|t\partial _tP_t^\lambda (f)(x)\|_\mathbb{B}^2\frac{dt}{t}\right)^{1/2},\quad x\in (0,\infty ),$$
for every $f\in L^p((0,\infty ),\mathbb{B})$, $1<p<\infty$. According to \cite[Theorems 2.4 and 2.5]{BFMT} and \cite{Kw}
we have that, $\mathbb{B}$ is isomorphic to a Hilbert space if, and only if, for some (equivalently, for every ) $1<p<\infty$,
there exists $C>0$ for which
$$\frac{1}{C}\|f\|_{L^p((0,\infty ),\mathbb{B})}
    \leq \|g_{1, \mathbb{B}}(\{P_t^\lambda \}_{t>0})(f)\|_{L^p(0,\infty )}
    \leq C \|f\|_{L^p((0,\infty ),\mathbb{B})},\quad f\in L^p((0,\infty ),\mathbb{B}).$$
Note that the semigroup $\{P_t^\lambda \}_{t>0}$ is not Markovian. Hence, the results in \cite{MTX} do not imply those in \cite{BFMT}.
Also the theory developed in \cite{Hy} do not apply for the Bessel Poisson semigroup.

In \cite{SW} Segovia and Wheeden defined a fractional derivative as follows. Suppose that
$F:\Omega \times (0,\infty )\longrightarrow \mathbb{C}$ is a good enough function, where
$\Omega \subset \mathbb{R}^n$, and $\beta >0$. The $\beta$-derivative $\partial_t^\beta F$ is defined by
$$\partial _t^\beta F(x,t)
    =\frac{e^{-i\pi (m-\beta)}}{\Gamma (m-\beta )}\int_0^\infty \partial_t^mF(x,t+s)s^{m-\beta -1}ds,\quad x\in \Omega, \ t\in (0,\infty ),$$
where $m\in \mathbb{N} $ and $m-1\leq \beta <m$. By using this fractional derivative, Segovia and Wheeden obtained characterizations of Sobolev spaces.

If $\mathbb{B}$ is a Banach space and $\beta >0$ we define the operator $G_{P,\mathbb{B}}^{\lambda ,\beta }$ by
$$G_{P,\mathbb{B}}^{\lambda ,\beta }(f)(x)
    =t^\beta \partial _t^\beta P_t^\lambda (f)(x),\quad t,x\in (0,\infty ),$$
for every $f\in L^p((0,\infty ),\mathbb{B})$, $1<p<\infty $.

We now prove that the operators $G_{P,\mathbb{B}}^{\lambda ,\beta }$ allow us
to get new equivalent norms in $L^p((0,\infty ),\mathbb{B})$ provided that $\mathbb{B}$ is a UMD space.
\begin{Th}\label{boundedness}
    Let $\mathbb{B}$ be a UMD Banach space, $\lambda ,\beta >0$ and $1<p<\infty$. Then, there exists $C>0$ such that
    \begin{equation}\label{7.1}
        \frac{1}{C}\|f\|_{L^p((0,\infty ),\mathbb{B})}
        \leq \|G_{P, \mathbb{B}}^{\lambda ,\beta }(f)\|_{L^p((0,\infty ),\gamma (H,\mathbb{B}))}
        \leq C \|f\|_{L^p((0,\infty ),\mathbb{B})},
    \end{equation}
    for every $f\in L^p((0,\infty ),\mathbb{B})$.
\end{Th}

For every $\lambda>0$ the Poisson semigroup $\{P_t^\lambda\}_{t>0}$ is generated by $-\sqrt{\Delta_\lambda}$ in
$L^p(0,\infty)$, $1<p<\infty$. According to \cite[Proposition 6.1]{NoSt3},
 $\{P_t^\lambda\}_{t>0}$ is contractive in
$L^p(0,\infty)$, $1<p<\infty$, provided that $\lambda \geq 1$. Then, by \cite[Theorem 6.1]{Tag}, equivalence \eqref{7.1}
follows from \cite[Proposition 2.16]{NVW2} (see also \cite[Lemma 2.3]{Maa}) when $\lambda \geq 1$, $\beta>0$, $1<p<\infty$
and $\B$ is a UMD Banach space. In Theorem~\ref{boundedness}, \eqref{7.1} is established for every $\lambda>0$. Our proof
(see Section~\ref{sec:Proof2}) does not use functional calculus arguments. We exploit the fact that the
Bessel operator $\Delta_\lambda$ is, in some sense, a nice perturbation of the Laplacian operator $-d^2/dx^2$.
We connect the $g$-function operator $G_{P, \mathbb{B}}^{\lambda ,\beta }$ with the corresponding operator associated
with the classical Poisson semigroup and then we apply Theorem~\ref{ThB}.

We also consider square functions associated with Bessel Poisson semigroups involving derivative with respect to $x$.
If $\mathbb{B}$ is a Banach space we define, for every $f\in L^p((0,\infty ),\mathbb{B})$, $1<p<\infty$,
$$ \mathcal{G}_{P,\mathbb{B}}^\lambda (f)(t,x)
    =tD_{\lambda ,x}^* P_t^{\lambda +1}(f)(x),\quad x,t\in (0,\infty ),$$
where $D_\lambda ^*=-x^{-\lambda }\frac{d}{dx}x^\lambda $.

\begin{Th}\label{boundedness2}
    Let $\mathbb{B}$ be a UMD Banach space, $\lambda >0$ and $1<p<\infty$. Then, the operator
    $\mathcal{G}_{P,\mathbb{B}}^\lambda $ is bounded from $L^p((0,\infty ),\mathbb{B})$ into
    $L^p((0,\infty ),\gamma(H,\mathbb{B}))$.
\end{Th}

The operators $G_{P,\mathbb{B}}^{\lambda ,1}$ and $\mathcal{G}_{P,\mathbb{B}}^\lambda$ are connected by certain Cauchy-Riemann
type equations and Riesz transforms associated with Bessel operators. These relations allow us to get new
characterizations of UMD Banach spaces. Also, the equivalence of $L^p$-norms in Theorem \ref{boundedness}
characterizes UMD Banach spaces. In order to see this last property we need first to describe UMD Banach
spaces by using $L^p$-boundedness of the imaginary power $\Delta_\lambda^{i\omega}$, $\omega \in \mathbb{R}\setminus\{0\}$,
of Bessel operators (see Proposition~\ref{ImagBess}).

\begin{Th}\label{Th1.4}
    Let $\mathbb{B}$ be a Banach space and $\lambda >0$. The following assertions are equivalent.

    (i) $\mathbb{B}$ is UMD.

    (ii) For some (equivalently, for every) $1<p<\infty$, there exists $C>0$ such that,
    \begin{equation}\label{A2}
        \frac{1}{C} \|f\|_{L^p((0,\infty ),\mathbb{B})}
            \leq \|G_{P, \mathbb{B}}^{\lambda ,1 }(f)\|_{L^p((0,\infty ),\gamma (H,\mathbb{B}))}, \quad f\in L^p(0,\infty )\otimes \mathbb{B},
    \end{equation}
    and
    \begin{equation}\label{A3}
        \|\mathcal{G}_{P, \mathbb{B}}^{\lambda }(f)\|_{L^p((0,\infty ),\gamma (H,\mathbb{B}))}
            \leq C \|f\|_{L^p((0,\infty ),\mathbb{B})},  \quad f\in L^p(0,\infty )\otimes \mathbb{B}.
    \end{equation}

    (iii) For some (equivalently, for every) $1<p<\infty$ and $\beta >0$, there exists $C>0$ such that, for $\delta =\beta$ and $\delta =\beta +1$,
    \begin{equation}\label{A1}
        \frac{1}{C}\|f\|_{L^p((0,\infty ),\mathbb{B})}
            \leq \|G_{P, \mathbb{B}}^{\lambda ,\delta }(f)\|_{L^p((0,\infty ),\gamma (H,\mathbb{B}))}
            \leq C \|f\|_{L^p((0,\infty ),\mathbb{B})},  \quad f\in L^p(0,\infty )\otimes \mathbb{B}.
    \end{equation}
\end{Th}

Inspired in \cite[Theorem 1]{Me1} as application of the result in Theorem \ref{boundedness}
we give sufficient conditions in order that spectral multipliers associated with Bessel
operators are bounded in $L^p((0,\infty ),\mathbb{B})$, $1<p<\infty$.

If $f\in \mathcal{S}_\lambda (0,\infty)$ from \eqref{BesselJ} we deduce that
$$h_\lambda (\Delta _\lambda f)(x)
    =x^2h_\lambda (f)(x),\quad x\in (0,\infty ).$$
We define
$$\Delta _\lambda f
    =h_\lambda (x^2h_\lambda (f)),\quad f\in D(\Delta _\lambda ),$$
where the domain $D(\Delta _\lambda )$ of $\Delta _\lambda $ is
$$D(\Delta _\lambda )
    =\{f\in L^2(0,\infty ): x^2h_\lambda (f)\in L^2(0,\infty )\}.$$
Suppose that $m\in L^\infty (0,\infty )$. The spectral multiplier $m(\Delta _\lambda )$ is defined by
\begin{equation}\label{M1}
    m(\Delta _\lambda )(f)
        =h_\lambda (m(x^2)h_\lambda ),\quad f\in L^2(0,\infty ).
\end{equation}
Since $h_\lambda $ is bounded in $L^2(0,\infty )$, it is clear that $m(\Delta _\lambda)$ is
bounded from $L^2(0,\infty )$ into itself. At this point the question is to give
conditions on the function $m$ which imply that the operator $m(\Delta _\lambda )$ can be
extended from $L^2(0,\infty )\cap L^p(0,\infty )$ to $L^p(0,\infty )$ as a bounded operator
from $L^p(0,\infty )$ into itself for some $p\in (1,\infty )\setminus\{2\}$.

In \cite{BCC1} and \cite{BMR} Laplace transform type Hankel multipliers were investigated. A function $m$ is called of Laplace transform type when
$$m(y)
    =y\int_0^\infty e^{-yt}\psi (t)dt,\quad y\in (0,\infty ),$$
for some $\psi\in L^\infty (0,\infty )$. If $m$ is of Laplace transform type then the operator
$m(\Delta _\lambda )$ defined in \eqref{M1} can be extended to $L^p(0,\infty )$ as a bounded operator
from $L^p(0,\infty )$ into itself, $1<p<\infty$, and from $L^1(0,\infty )$ into $L^{1,\infty }(0,\infty )$ (\cite{BCC1}, \cite{BMR} and \cite[p. 121]{Ste1}).

Let $\omega \in \mathbb{R}\setminus\{0\}$. The imaginary power $\Delta_\lambda ^{i\omega}$ of $\Delta _\lambda$ is defined by
$$\Delta_\lambda^{i\omega}(f)
    =h_\lambda (y^{2i \omega}h_\lambda (f)),\quad f\in L^2(0,\infty ).$$
Since
$$y^{ i \omega}
    =y\int_0^\infty e^{-yt}\frac{t^{-i\omega}}{\Gamma (1-i \omega )}dt, \quad y\in (0,\infty ),$$
the operator $\Delta_\lambda^{i\omega}$ is a Laplace transform type Hankel multiplier.

In Proposition~\ref{ImagBess} (Section~\ref{sec:Proof4}) we show that a Banach space $\mathbb{B}$ is
UMD if, and only if, the operator $\Delta_\lambda^{i\omega}$, $\omega \in \mathbb{R}$,
is a bounded operator from $L^p((0,\infty ),\mathbb{B})$ into itself, for some (equivalently, for every) $1<p<\infty$.
This is a Bessel version of \cite[Theorem, p. 402]{Gue}.

In the following theorem we establish a Banach valued version of \cite[Theorem 1]{Me1} for the Bessel operator.

If $m\in L^\infty (0,\infty )$ we define, for every $n\in \mathbb{N}$,
$$m_n(t,y)
    =(ty)^ne^{-ty/2}m(y^2),\quad t,y \in (0,\infty ),$$
and $\mathcal{M}_n(t,u)$, $t\in (0,\infty )$, $u\in\mathbb{R}$, represents the Mellin transform of $m_n$ with respect to the variable $y$, that is,
$$\mathcal{M}_n(t,u)
    =\int_0^\infty m_n(t,y)y^{-iu-1}dy,\quad u\in \mathbb{R} \text{ and } t>0.$$

\begin{Th}\label{Th1.5}
    Let $\mathbb{B}$ be a UMD Banach space, $\lambda >0$ and $m \in L^\infty(0,\infty)$.
    Suppose that for some $1<p<\infty$ and $n\in \mathbb{N}$ the following property holds
    \begin{equation}\label{M2}
        \int_\R \sup _{t>0}|\mathcal{M}_n(t,u)|\|\Delta_\lambda ^{iu/2}\|_{L^p((0,\infty ),\mathbb{B})\rightarrow L^p((0,\infty ),\mathbb{B})}du<\infty .
    \end{equation}
    Then, $m(\Delta _\lambda )$ can be extended from $\mathcal{S}_\lambda (0,\infty )\otimes \mathbb{B}$ to $L^p((0,\infty ),\mathbb{B})$
    as a bounded operator from $L^p((0,\infty ),\mathbb{B})$ into itself.
\end{Th}

We now specify some conditions over the function $m$ and the UMD Banach space $\mathbb{B}$ for which \eqref{M2} is satisfied.
As in \cite[Theorem 3]{Me1} we consider $m\in L^\infty (0,\infty )$ that extends to a bounded analytic function in a sector
$\Sigma_\vartheta =\{z\in \mathbb{C}:|\mbox{Arg }\;z|<\vartheta \}$. In this case, we have that
$$\sup _{t>0}|\mathcal{M}_n(t,u)|
    \leq Ce^{\pi |u|/2}(1+|u|),\quad u\in \mathbb{R}.$$
By \cite[Corollary 1]{Cow} (see also \cite[Corollary 1.2]{BCC1}) we can obtain that, for every $1<p<\infty$,
\begin{equation}\label{M3}
    \|\Delta _\lambda ^{iu}\|_{L^p(0,\infty )\rightarrow L^p(0,\infty )}
        \leq C(1+|u|^3\log  |u|)^{|1/p-1/2|}\exp \left(\pi |1/p-1/2||u|\right),\quad u\in \mathbb{R},
\end{equation}
where $C>0$ depends on $p$ but it does not depend on $u$.

Even when we consider the usual Laplacian operator instead of the Bessel operator
$\Delta _\lambda$, it is not known if \eqref{M3} holds when the functions take values in a UMD Banach
space (see, for instance, \cite[Corollary 2.5.3]{TagPhD}). In order to get an estimate as \eqref{M3},
replacing $L^p(0,\infty )$ by $L^p((0,\infty ),\mathbb{B})$ we need to strengthen the property of the
Banach spaces as follows. $\mathbb{B}$ must be isomorphic to a closed subspace of
a complex interpolation space $[\HH,X]_\theta$, where $0<\theta <1$, $\HH$ is a Hilbert space and $X$ is a UMD Banach space.
When $\mathbb{B}$ satisfies this property for some $\theta \in (0,1)$ we write $\B\in I_\theta (\mathpzc{H}, UMD)$.
The class $\cup _{\theta \in (0,1)}I_\theta (\mathpzc{H}, UMD)$ includes all UMD lattices (\cite[Corollary on p. 216]{Rub})
and it also includes the Schatten ideals $C_p$, $p\in (1,\infty )$ (see \cite{DDP}).
It is clear that $\mathbb{B}$ is UMD provided that $\B\in \cup _{\theta \in (0,1)}I_\theta (\mathpzc{H}, UMD)$.

As far as it is known, it is an open problem whether every UMD Banach space is in
$\cup _{\theta \in (0,1)}$ $I_\theta (\mathpzc{H}, UMD)$ (\cite[Problem 4 on p. 220]{Rub}).
This class of Banach spaces has been used, for instance,
in \cite{Hy}, \cite{MTX} and \cite{TagPhD}, and also it plays a central role in the vector-valued version of Carleson's
theorem recently established in \cite{HyLa}.

\begin{Th}\label{Th1.6}
    Let $\lambda \geq 1$. Suppose that $m$ is a bounded holomorphic function in
    $\Sigma_\vartheta$ for certain $\vartheta \in (0,\pi )$, and that the Banach space $\mathbb{B}$ is in $I_\theta (\mathpzc{H}, UMD)$,
    for some $\theta \in (0,\vartheta/\pi)$. Then, the spectral multiplier $m(\Delta _\lambda )$ can be extended to
    $L^q((0,\infty ),\mathbb{B})$ as a bounded operator from
    $L^q((0,\infty ),\mathbb{B})$ into itself, for every $q\in [2/(1+\theta), 2/(1-\theta)]$.
\end{Th}

In the following sections of this paper we present proofs for our theorems. Throughout this paper $C$ and $c$ always
denote positive constants, not necessarily the same in each occurrence.

\section{Proof of Theorem~\ref{ThBHankel}}\label{sec:Proof1}

\subsection{}\label{subsec:2.1}
Firstly we  prove that there exists $C>0$ such that
\begin{equation}\label{Lp1}
    \|\mathcal{W}_{\psi ,\mathbb{B}}^\lambda (f)\|_{L^p((0,\infty ),\gamma (H, \mathbb{B}))}
        \leq C\|f\|_{L^p((0,\infty ),\mathbb{B})},
\end{equation}
for every $f\in L^p((0,\infty ),\mathbb{B})$.

We choose $\phi \in\mathcal{S}(\mathbb{R})$ such that $\phi (x^2)= x^{-\lambda }\psi (x)$, $x\in (0,\infty )$
(see \cite[p. 85]{EG}). Then, we can write
\begin{align*}
    \;_\lambda \tau _x(\psi _{(t)} )(y)
        = &\frac{(xy)^\lambda t^{-\lambda -1}}{\sqrt{\pi }2^{\lambda -1/2}\Gamma (\lambda )}\int_0^\pi
                \psi \left(\frac{\sqrt{(x-y)^2+2xy(1-\cos \theta )}}{t}\right)\\
          & \times ((x-y)^2+2xy (1-\cos \theta ))^{-\lambda /2}(\sin \theta )^{2\lambda -1}d\theta \\
        = &\frac{(xy)^\lambda t^{-2\lambda -1}}{\sqrt{\pi }2^{\lambda -1/2}\Gamma (\lambda)}\int_0^\pi (\sin \theta )^{2\lambda -1}
                \phi \left(\frac{(x-y)^2+2xy(1-\cos \theta )}{t^2}\right)d\theta ,\quad t,x,y   \in (0,\infty ).
\end{align*}

We define the function $\Phi $ as follows:
$$\Phi (x)
    =\frac{1}{\sqrt{\pi }2^{\lambda +1/2}\Gamma (\lambda )}\int_0^\infty u^{\lambda -1}\phi (x^2+u)du,\qquad x\in \mathbb{R}.$$

It is not hard to see that $\Phi \in \mathcal{S}(\mathbb{R})$. Hence, since $\widehat{\Phi}(0)=0$ (see \cite[(17)]{BCR1}),
$\Phi $ satisfies conditions $(C1)$ and $(C2)$ in \cite[p. 111]{KaWe}
($(i)$ and $(ii)$ in Theorem~\ref{ThA}).

We consider the operator
$$\mathcal{W}_{\Phi ,\mathbb{B}}(f)(t,x)
    =(\Phi _t*f)(x),\quad f\in L^p(\mathbb{R},\mathbb{B}), \ t\in (0,\infty )\mbox{ and }x\in \mathbb{R}.$$

According to \cite[Theorem 4.2]{KaWe} (Theorem~\ref{ThB}) we have that, for every $f\in \mathcal{S}(\mathbb{R})\otimes \mathbb{B}$,
\begin{equation}\label{Lp2}
    \|\mathcal{W}_{\Phi ,\mathbb{B}}(f)\|_{L^p(\mathbb{R},\gamma (H,\mathbb{B}))}
        \leq C\|f\|_{L^p(\mathbb{R},\mathbb{B})}.
\end{equation}
 We are going too see that the inequality \eqref{Lp2} holds for every $f\in L^p(\mathbb{R},\mathbb{B})$.
 Let $f\in L^p(\mathbb{R},\mathbb{B})$. We choose a sequence
 $(f_n)_{n\in \mathbb{N}}\subset \mathcal{S}(\mathbb{R})\otimes \mathbb{B}$ such that
 $f_n\longrightarrow f$, as $n\rightarrow \infty$, in $L^p(\mathbb{R},\mathbb{B})$. According to \eqref{Lp2}, by defining
$$\widetilde{\mathcal{W}}_{\Phi ,\mathbb{B}}(f)
    =\lim_{n\rightarrow \infty }\mathcal{W}_{\Phi ,\mathbb{B}}(f_n),$$
where the limit is understood in $L^p(\mathbb{R},\gamma (H,\mathbb{B}))$, we have that
$$\|\widetilde{\mathcal{W}}_{\Phi ,\mathbb{B}}(f)\|_{L^p(\mathbb{R},\gamma (H,\mathbb{B}))}
    \leq C\|f\|_{L^p(\mathbb{R},\mathbb{B})}.$$
Also, there exists an increasing sequence $(n_k)_{k\in \mathbb{N}}\subset \mathbb{N}$ and a subset $\Omega$ of $\mathbb{R}$ such that
$$\widetilde{\mathcal{W}}_{\Phi ,\mathbb{B}}(f)(x)
    =\lim_{k \rightarrow \infty }\mathcal{W}_{\Phi ,\mathbb{B}}(f_{n_k})(\cdot,x),\quad x\in \Omega ,$$
where the limit is understood in $\gamma (H,\mathbb{B})$, and $|\mathbb{R}\setminus \Omega |=0$.

For every $\varepsilon >0$,
$$\mathcal{W}_{\Phi ,\mathbb{B}}(f_n)(\cdot,x)\longrightarrow \mathcal{W}_{\Phi ,\mathbb{B}}(f)(\cdot,x), \quad \text{as } n\rightarrow \infty,
    \text{ in } L^2((\varepsilon ,\infty ), dt/t, \mathbb{B}),$$
uniformly in $x\in \mathbb{R}$. Indeed, let $\varepsilon >0$. By using Minkowski's inequality we get
\begin{align*}
    & \left(\int_\varepsilon ^\infty \|\mathcal{W}_{\Phi ,\mathbb{B}}(f_n)(t,x)-\mathcal{W}_{\Phi , \mathbb{B}}(f)(t,x)\|_\mathbb{B}^2\frac{dt}{t} \right)^{1/2}\\
    & \qquad \leq \int_\mathbb{R}\|f_n(y)-f(y)\|_\mathbb{B}\left(\int_\varepsilon ^\infty \frac{1}{t^3}\Big|\Phi \Big(\frac{|x-y|}{t}\Big)\Big|^2dt\right)^{1/2}dy\\
    & \qquad \leq C\int_\mathbb{R}\|f_n(y)-f(y)\|_\mathbb{B}\left(\int_\varepsilon ^\infty \frac{1}{(t+|x-y|)^3}dt\right)^{1/2}dy\\
    & \qquad \leq C\int_\mathbb{R}\frac{\|f_n(y)-f(y)\|_\mathbb{B}}{\varepsilon +|x-y|}dy
        \leq C\|f_n-f\|_{L^p(\mathbb{R},\mathbb{B})}\left(\int_\mathbb{R}\frac{dy}{(\varepsilon +|y|)^{p'}}\right)^{1/p'} \\
    & \qquad \leq C\varepsilon ^{-1/p}\|f_n-f\|_{L^p(\mathbb{R},\mathbb{B})},\quad n\in \mathbb{N}\mbox{ and }x\in \mathbb{R},
\end{align*}
where $p'$ is the conjugated exponent of $p$, that is, $p'=p/(p-1)$.

Let $S\in \mathbb{B}^*$. Since $\gamma (H,\mathbb{B})$ is continuously contained in the space $L(H,\mathbb{B})$ of
linear bounded operators from $H$ into $\mathbb{B}$, for every $x\in \Omega $ and $h\in L^2((0,\infty ),dt/t)$
with $\mbox{supp}\; (h)\subset (0,\infty )$, we have that
\begin{align*}
    \langle S, [\widetilde{\mathcal{W}}_{\Phi ,\mathbb{B}}(f)(x)](h)\rangle _{\mathbb{B}^ *,\mathbb{B}}
        =&\lim_{k\rightarrow \infty  }\langle S,[\mathcal{W}_{\Phi ,\mathbb{B}}(f_{n_k})(\cdot ,x)](h)\rangle _{\mathbb{B}^ *,\mathbb{B}}\\
        =&\lim_{k\rightarrow \infty}\int_0^\infty \langle S,\mathcal{W}_{\Phi ,\mathbb{B}}(f_{n_k})(t,x)\rangle _{\mathbb{B}^ *,\mathbb{B}}h(t)\frac{dt}{t}\\
        =&\int_0^\infty \langle S, \mathcal{W}_{\Phi ,\mathbb{B}}(f)(t,x)\rangle _{\mathbb{B}^ *,\mathbb{B}}h(t)\frac{dt}{t}.
\end{align*}
Hence, for every $x\in \Omega $, $\langle S, \mathcal{W}_{\Phi ,\mathbb{B}}(f)(\cdot , x)\rangle _{\mathbb{B}^ *,\mathbb{B}}\in L^2((0,\infty ),dt/t)$ and
$$ \langle S, [\widetilde{\mathcal{W}}_{\Phi ,\mathbb{B}}(f)(x)](h)\rangle _{\mathbb{B}^ *,\mathbb{B}}
    =\int_0^\infty \langle S, \mathcal{W}_{\Phi ,\mathbb{B}}(f)(t,x) \rangle_{\B^*,\B} h(t)\frac{dt}{t},\quad h\in H.$$
We conclude that $\widetilde{\mathcal{W}}_{\Phi ,\mathbb{B}}(f)(x)=\mathcal{W}_{\Phi ,\mathbb{B}}(f)(\cdot,x)$, $x\in \Omega$, as elements of $\gamma (H,\mathbb{B})$.
Hence, \eqref{Lp2} holds for every $f \in L^p(\R,\B)$.

Suppose now that $f\in L^p((0,\infty ),\mathbb{B})$. By defining the function $f_o$ as the odd extension of $f$ to $\mathbb{R}$, we have that
\begin{align*}
    \mathcal{W}_{\Phi ,\mathbb{B}}(f_o)(t,x)
        = &\frac{1}{t}\int_{-\infty }^{+\infty }\Phi \Big( \frac{x-y}{t}\Big)f_o(y)dy \\
        = -& \frac{1}{t}\int_0^\infty  \Phi \Big( \frac{x+y}{t}\Big)f(y)dy+\frac{1}{t}\int_0^\infty \Phi \Big( \frac{x-y}{t}\Big)f(y)dy \\
        = &L_{\Phi ,\mathbb{B}}^1(f)(t,x)+L_{\Phi ,\mathbb{B}}^2(f)(t,x),\quad x\in \mathbb{R} \mbox{ and } t\in (0,\infty ).
\end{align*}
We get by using Minkowski's inequality
\begin{align*}
    & \|L_{\Phi ,\mathbb{B}}^1(f)(\cdot ,x)\|_{L^2((0,\infty ),dt/t,\mathbb{B})}
        \leq \int_0^\infty \|f(y)\|_\mathbb{B}\left(\int_0^\infty \frac{1}{t^3}\Big|\Phi \Big(\frac{x+y}{t}\Big)\Big|^2dt\right)^{1/2}dy\\
        & \qquad \leq C\int_0^\infty \frac{\|f(y)\|_\mathbb{B}}{x+y}dy
        \leq C \left[H_0(\|f\|_\mathbb{B})(x)+H_\infty (\|f\|_\mathbb{B})(x) \right],\quad x\in (0,\infty ),
\end{align*}
where $H_0$ and $H_\infty $ denote the Hardy operators defined by
$$H_0(g)(x)
    =\frac{1}{x}\int_0^xg(y)dy,\quad x\in (0,\infty ),$$
and
$$H_\infty (g)(x)
    =\int_x^\infty \frac{g(y)}{y}dy, \quad x\in (0,\infty ).$$
Since $H_0$ and $H_\infty$ are bounded operators from $L^p(0,\infty )$ into itself (\cite[p. 244, (9.9.1) and (9.9.2)]{HLP}),
$L^1_{\Phi ,\mathbb{B}}$ is a bounded operator from $L^p((0,\infty ),\mathbb{B})$ into $L^p((0,\infty ),L^2((0,\infty ),dt/t,\mathbb{B}))$.
By taking into account that $\gamma (H,\mathbb{C})=H$, we deduce that, for every $x\in (0,\infty )$,
$L_{\Phi ,\mathbb{B}}^1(f)(\cdot ,x)\in \gamma (H,\mathbb{B})$ and $L_{\Phi ,\mathbb{B}}^1$ defines a bounded
operator from $L^p((0,\infty ),\mathbb{B})$ into $L^p((0,\infty ),\gamma (H,\mathbb{B}))$.

By using now \eqref{Lp2} we conclude that the operator $L^2_{\Phi ,\mathbb{B}}$ is bounded from
$L^p((0,\infty ),\mathbb{B})$ into $L^p((0,\infty ),\gamma (H,\mathbb{B}))$.

Inequality \eqref{Lp1} will be proved once we establish that
\begin{equation}\label{11.1}
    \left\| \left[ \mathcal{W}_{\psi,\B}^\lambda -  L^2_{\Phi ,\mathbb{B}} \right](f)\right\|_{L^p((0,\infty),\gamma(H,\B))}
        \leq C \|f\|_{L^p((0,\infty),\B)}, \quad f \in L^p((0,\infty),\B).
\end{equation}
In order to do this, we study the function
$$K_\lambda (t,x,y)
    =\;_\lambda \tau _x(\psi _{(t)} )(y)-\Phi _t(x-y),\quad t,x,y\in (0,\infty ).$$
Firstly we write
$$\;_\lambda \tau _x(\psi _{(t)} )(y)
    =H_{\lambda ,1}(t,x,y)+H_{\lambda ,2}(t,x,y),\quad t,x,y\in (0,\infty ),$$
where
$$H_{\lambda ,1}(t,x,y)
    =\frac{(xy)^\lambda t^{-2\lambda -1}}{\sqrt{\pi }2^{\lambda -1/2}\Gamma (\lambda )}\int_0^{\pi /2}(\sin \theta )^{2\lambda -1}
            \phi \left(\frac{(x-y)^2+2xy(1-\cos \theta )}{t^2}\right)d\theta, \quad t,x,y\in (0,\infty ).$$
We have that, for every $x \in (0,\infty)$,
\begin{align*}
    & \|H_{\lambda ,2}(\cdot,x,y) \|_H \\
    & \qquad \leq C(xy)^\lambda \left(\int_0^\infty t^{-4\lambda -3}\left(\int_{\pi /2}^\pi
        (\sin \theta )^{2\lambda -1}\left|\phi \left(\frac{(x-y)^2+2xy(1-\cos \theta )}{t^2}\right)\right|d\theta\right)^2 dt\right)^{1/2}\\
    & \qquad \leq C(xy)^\lambda
        \left\{\begin{array}{l}
            \displaystyle \frac{1}{|x-y|^{2\lambda +1}}  \Big(\int_0^\infty u^{-4\lambda -3} \Big(\int_{\pi /2}^\pi (\sin \theta )^{2\lambda -1}\\
            \displaystyle \qquad \times \left|\phi \left(\frac{(x-y)^2+2xy(1-\cos \theta )}{(x-y)^2u^2}\right)\right|d\theta\Big)^2 du\Big)^{1/2}, \quad y\in (0,x/2)\cup (2x,\infty ),\\
            \\
            \displaystyle \frac{1}{(xy)^{\lambda +1/2}} \Big(\int_0^\infty u^{-4\lambda -3}\Big(\int_{\pi /2}^\pi (\sin \theta )^{2\lambda -1}\\
            \displaystyle \qquad \times \left|\phi \left(\frac{(x-y)^2+2xy(1-\cos \theta )}{xyu^2}\right)\right|d\theta\Big)^2 du\Big)^{1/2}, \quad y\in (x/2,2x).
        \end{array} \right.
\end{align*}
Then, since $\phi \in \mathcal{S}(\mathbb{R})$, it follows that
\begin{align}\label{Lp3}
    & \|H_{\lambda ,2}(\cdot,x,y)\|_H
      \leq C \frac{(xy)^\lambda }{|x-y|^{2\lambda +1}}\left(\int_1^\infty u^{-4\lambda -3} du+\int_0^1du\right)^{1/2}
      \leq C\left\{\begin{array}{ll}
        \displaystyle \dfrac{1}{x},&0<y<x/2,\\
        &\\
        \displaystyle \dfrac{1}{y},&y>2x>0,
        \end{array}\right.
\end{align}
and
\begin{equation}\label{Lp4}
    \|H_{\lambda ,2}(\cdot,x,y)\|_H
        \leq C\frac{1}{(xy)^{1/2}}\left(\int_1^\infty u^{-4\lambda -3} du+\int_0^1du\right)^{1/2}\leq \frac{C}{x}, \quad y\in (x/2,2x).
\end{equation}
By proceeding in a similar way we can see that
\begin{equation}\label{Lp5}
    \|H_{\lambda ,1}(\cdot,x,y)\|_H
        \leq C\left\{\begin{array}{ll}
        \displaystyle \frac{1}{x},&0<y<x/2,\\
        &\\
        \displaystyle \frac{1}{y},&y>2x>0,
        \end{array}\right.
\end{equation}
and also
\begin{equation}\label{Lp6}
    \left\|\Phi_t(x-y)\right\|_H\leq \frac{C}{|x-y|}\leq C\left\{\begin{array}{ll}
        \displaystyle \frac{1}{x},&0<y<x/2,\\
        &\\
        \displaystyle \frac{1}{y},&y>2x
        \end{array}\right., \quad x\in (0,\infty ).
\end{equation}
Suppose now that $x\in (0,\infty )$ and $x/2<y<2x$. We split the difference $H_{\lambda ,1}(t,x,y)-\Phi _t(x-y)$, $t\in (0,\infty )$, as follows
\begin{align*}
    & H_{\lambda ,1}(t,x,y)-\Phi _t(x-y)
        =\frac{(xy)^\lambda t^{-2\lambda -1}}{\sqrt{\pi }2^{\lambda -1/2}\Gamma (\lambda )}\int_0^{\pi /2}
                [(\sin \theta )^{2\lambda -1}-\theta ^{2\lambda -1}]\phi \left(\frac{(x-y)^2+2xy(1-\cos \theta )}{t^2}\right)d\theta\\
    & \qquad \qquad + \frac{(xy)^\lambda t^{-2\lambda -1}}{\sqrt{\pi }2^{\lambda -1/2}\Gamma (\lambda )}\int_0^{\pi /2}
            \theta ^{2\lambda -1}\left[\phi \left(\frac{(x-y)^2+2xy(1-\cos \theta )}{t^2}\right)-
                \phi \left(\frac{(x-y)^2+xy\theta ^2}{t^2}\right)\right]d\theta\\
    & \qquad \qquad +\frac{(xy)^\lambda t^{-2\lambda -1}}{\sqrt{\pi }2^{\lambda -1/2}\Gamma (\lambda )}\int_0^{\pi /2}\theta ^{2\lambda -1}
        \phi \left(\frac{(x-y)^2+xy\theta ^2}{t^2}\right)d\theta-\Phi _t(x-y)\\
    & \qquad = H_{\lambda ,1,1}(t,x,y)+H_{\lambda ,1,2}(t,x,y) +H_{\lambda ,1,3}(t,x,y), \quad t,x,y\in (0,\infty ).
\end{align*}
By using the mean value theorem we get
\begin{align*}
    & \|H_{\lambda ,1,1}(\cdot,x,y)\|_H\\
    & \qquad \leq C(xy)^\lambda \left\{\int_0^\infty t^{-4\lambda -3}\left(\int_0^{\pi /2}\theta ^{2\lambda +1}\left|\phi \left(\frac{(x-y)^2+2xy(1-\cos \theta )}{t^2}\right)\right| d\theta \right)^2dt\right\}^{1/2}\\
    & \qquad \leq \frac{C}{(xy)^{1/2}} \left\{\int_0^\infty u^{-4\lambda -3}\left(\int_0^{\pi /2}\theta ^{2\lambda +1}\left|\phi \left(\frac{(x-y)^2+2xy(1-\cos \theta )}{u^2xy}\right)\right| d\theta \right)^2du\right\}^{1/2}\\
    & \qquad \leq \frac{C}{(xy)^{1/2}} \left\{\int_1^\infty  u^{-4\lambda -3}du+\int_0^1 u^{-4\lambda -3}\left(\int_0^{\pi /2}\theta ^{2\lambda +1}\left(\frac{u^2xy}{(x-y)^2+xy \theta^2}\right)^{\lambda +3/4}d\theta \right)^2du\right\}^{1/2}\\
    & \qquad \leq \frac{C}{x},
\end{align*}
and
\begin{align*}
    & \|H_{\lambda ,1,2}(\cdot,x,y)|\|_H\\
    & \quad \leq \frac{C}{(xy)^{1/2}} \left\{\int_0^\infty u^{-4\lambda -7}\left(\int_0^{\pi /2}\theta ^{2\lambda -1}\left|\int_{1-\cos \theta}^{\theta ^2/2}\phi '\left(\frac{(x-y)^2+2xyz}{u^2xy}\right)dz\right| d\theta \right)^2du\right\}^{1/2}\\
    &\quad \leq \frac{C}{(xy)^{1/2}} \left\{\int_1^\infty u^{-4\lambda -7}du
            +\int_0^1 u^{-4\lambda -7}\left(\int_0^{\pi /2}\theta ^{2\lambda -1}\int_{1-\cos \theta}^{\theta ^2/2}\left(\frac{u^2xy}{(x-y)^2+2xyz}\right)^{\lambda +7/4}dzd\theta \right)^2du\right\}^{1/2}\\
    &\quad \leq \frac{C}{(xy)^{1/2}} \left\{1+\int_0^1 \left(\int_0^{\pi /2}\theta ^{2\lambda -1}\left|\int_{1-\cos \theta}^{\theta ^2/2}\frac{dz}{z^{\lambda +7/4}}\right|d\theta \right)^2du\right\}^{1/2}\leq\frac{C}{x},
\end{align*}
On the other hand, a suitable change of variables allows us to write
\begin{align*}
     H_{\lambda ,1,3}(t,x,y)
        = & \frac{(xy)^\lambda t^{-2\lambda -1}}{\sqrt{\pi }2^{\lambda -1/2}\Gamma (\lambda )}\int_0^{\pi /2}\theta ^{2\lambda -1}
            \phi \left(\frac{(x-y)^2+xy\theta ^2}{t^2}\right)d\theta\\
        & - \frac{1}{t\sqrt{\pi }2^{\lambda +1/2}\Gamma (\lambda )}\int_0^\infty u^{\lambda -1}\phi \left(\Big(\frac{x-y}{t}\Big)^2+u\right)du\\
        = & -\frac{(xy)^\lambda t^{-2\lambda -1}}{\sqrt{\pi }2^{\lambda -1/2}\Gamma (\lambda )}\int_{\pi /2}^\infty \theta ^{2\lambda -1}\phi \left(\frac{(x-y)^2+xy\theta ^2}{t^2}\right)d\theta,\quad t\in (0,\infty ).
\end{align*}
Hence, we deduce that
\begin{align*}
    & \|H_{\lambda ,1,3}(\cdot,x,y)\|_H
        \leq \frac{C}{(xy)^{1/2}} \left\{\int_0^\infty u^{-4\lambda -3}\left(\int_{\pi /2}^\infty \theta ^{2\lambda -1}
            \left|\phi \left(\frac{(x-y)^2+xy\theta ^2}{xyu^2}\right)\right|d\theta\right)^2du\right\}^{1/2}\\
    & \qquad \leq \frac{C}{(xy)^{1/2}} \left\{\int_1^\infty u^{-4\lambda -3}\left(\int_{\pi /2}^\infty \theta ^{2\lambda -1}
        \left(\frac{xyu^2}{(x-y)^2+xy\theta ^2}\right)^{\lambda +1/4}d\theta\right)^2du\right.\\
    & \qquad \qquad + \left.\int_0^1u^{-4\lambda -3}\left(\int_{\pi /2}^\infty \theta ^{2\lambda -1}
        \left(\frac{xyu^2}{(x-y)^2+xy\theta ^2}\right)^{\lambda +3/4}d\theta\right)^2du\right\}^{1/2}
        \leq \frac{C}{x}.
\end{align*}
By putting together the above estimates we obtain
\begin{equation}\label{Lp7}
    \|H_{\lambda ,1}(\cdot ,x,y)-\Phi_t (x-y)\|_H \leq \frac{C}{x},\quad 0<\frac{x}{2}<y<2x.
\end{equation}
From \eqref{Lp3}-\eqref{Lp7} we deduce that
\begin{equation}\label{Lp8}
    \|K_\lambda (\cdot,x,y)\|_H
        \leq \frac{C}{\max \{x,y\}},\quad x,y\in (0,\infty ).
\end{equation}
By proceeding as in the case of $L^1_{\Phi,\B}$, since $\gamma (H,\mathbb{C})=H$, we infer from \eqref{Lp8} that the operator
$\mathcal{W}_{\psi ,\mathbb{B}}^\lambda -L_{\Phi ,\mathbb{B}}^2$ is bounded from $L^p((0,\infty ),\mathbb{B})$
into $L^p((0,\infty ),\gamma (H,\mathbb{B}))$. Thus \eqref{11.1} is established.

\subsection{}\label{subsec:2.2}
Our next objective is to show that there exists $C>0$ such that
\begin{equation}\label{Lp9}
    \|f\|_{L^p((0,\infty ),\mathbb{B})}
        \leq C\|\mathcal{W}_{\psi ,\mathbb{B}}^\lambda (f)\|_{L^p((0,\infty ),\gamma (H,\mathbb{B}))},
\end{equation}
for every $f\in L^p((0,\infty ),\mathbb{B})$. It is enough to see \eqref{Lp9}
for every $f\in\mathcal{S}_\lambda (0,\infty )\otimes \mathbb{B}$. Indeed, suppose that \eqref{Lp9}
is true for every $f\in \mathcal{S}_\lambda(0,\infty )\otimes \mathbb{B}$. Let $f\in L^p((0,\infty ),\mathbb{B})$.
We choose a sequence  $(f_n)_{n\in \mathbb{N}}\subset \mathcal{S}_\lambda(0,\infty )\otimes \mathbb{B}$ such that
$f_n\longrightarrow f$, as $n\rightarrow \infty$, in $L^p((0,\infty ),\mathbb{B})$. Then, by \eqref{Lp9}
\begin{equation}\label{Lp10}
    \|f_n\|_{L^p((0,\infty ),\mathbb{B})}
        \leq C\|\mathcal{W}_{\psi ,\mathbb{B}}^\lambda (f_n)\|_{L^p((0,\infty ),\gamma (H,\mathbb{B}))}, \quad n \in \N.
\end{equation}
Since, as it was proved in Section~\ref{subsec:2.1}, $\mathcal{W}_{\psi ,\mathbb{B}}^\lambda$ is a bounded operator
from $L^p((0,\infty ),\mathbb{B})$ into $L^p((0,\infty ),\gamma (H,\mathbb{B}))$, by
letting $n\rightarrow \infty$ in \eqref{Lp10} we conclude that
$$
\|f\|_{L^p((0,\infty ),\mathbb{B})}
\leq C\|\mathcal{W}_{\psi ,\mathbb{B}}^\lambda (f)\|_{L^p((0,\infty ),\gamma (H,\mathbb{B}))}.
$$

The following result was established in \cite[after Lemma 2.4]{BCR1}.
\begin{Lem}\label{Lema1}
    Let $\lambda >0$. If $\psi \in \mathcal{S}_\lambda (0,\infty )$ is not identically zero, then there exists
    $\phi\in \mathcal{S}_\lambda (0,\infty )$ such that
    \begin{equation}\label{Lp10b}
        \int_0^\infty h_\lambda (\psi)(y)h_\lambda (\phi )(y)y^{-2\lambda -1}dy=1,
    \end{equation}
    where the last integral is absolutely convergent.
\end{Lem}

In order to see \eqref{Lp9} we need to show the next result.
\begin{Lem}\label{Lema2}
    Let $\lambda >0$. Suppose that $\psi ,\phi \in \mathcal{S}_\lambda (0,\infty )$ satisfy \eqref{Lp10b},
    being the integral absolutely convergent. If $f,g\in \mathcal{S}_\lambda (0,\infty )$, then,
    \begin{equation}\label{Lp11}
        \int_0^\infty f(x)g(x)dx
            =\int_0^\infty \int_0^\infty (f\#_\lambda \psi _{(t)} )(y)(g\#_\lambda \phi_{(t)} )(y)\frac{dydt}{t}.
    \end{equation}
\end{Lem}

\begin{proof}
    Let $f,g\in \mathcal{S}_\lambda (0,\infty )$.
    Note firstly that the integral in the right hand side of \eqref{Lp11} is absolutely convergent. Indeed, according to \eqref{Lp1} we get
    \begin{align*}
        \int_0^\infty \int_0^\infty |(f\#_\lambda \psi _{(t)} )(y)||(g\#_\lambda \phi _{(t)} )(y)|\frac{dydt}{t}
            &\leq \|\mathcal{W}_{\psi ,\mathbb{C}}^\lambda (f)\|_{L^p((0,\infty),H )} \|\mathcal{W}_{\phi ,\mathbb{C}}^\lambda (g)\|_{L^{p'}((0,\infty), H )}\\
            & \leq C\|f\|_{L^p(0,\infty )}\|g\|_{L^{p'}(0,\infty )}.
    \end{align*}

    Plancherel equality and the interchange formula for Hankel transforms \eqref{6.1} lead to
    \begin{align*}
        \int_0^\infty (f\#_\lambda \psi _{(t)} )(y)(g\#_\lambda \phi _{(t)} )(y)dy
            & =\int_0^\infty h_\lambda (f\#_\lambda \psi _{(t)} )(y)h_\lambda (g\#_\lambda \phi _{(t)} )(y)dy\\
            & =\int_0^\infty h_\lambda (f)(y)h_\lambda (g)(y)(ty)^{-2\lambda }h_  \lambda (\psi)(ty)h_\lambda (\phi)(ty)dy,\quad t\in (0,\infty ).
    \end{align*}
    Hence, it follows that
    \begin{align*}
        & \int_0^\infty \int_0^\infty (f\#_\lambda \psi _{(t)} )(y)(g\#_\lambda \phi _{(t)} )(y)\frac{dydt}{t} \\
        & \qquad = \int_0^\infty \int_0^\infty h_\lambda (f)(y)h_\lambda (g)(y)(ty)^{-2\lambda }h_\lambda (\psi)(ty)h_\lambda (\phi)(ty)\frac{dydt}{t}\\
        & \qquad = \int_0^\infty h_\lambda (f)(y)h_\lambda (g)(y)\int_0^\infty h_\lambda (\psi)(ty)h_\lambda (\phi)(ty)(ty)^{-2\lambda } \frac{dtdy}{t}
            = \int_0^\infty h_\lambda (f)(y)h_\lambda (g)(y)dy\\
        & \qquad =\int_0^\infty f(x)g(x)dx.
    \end{align*}
\end{proof}
An immediate consequence of Lemma \ref{Lema2} is the following.
\begin{Lem}\label{Lema3}
    Let $\B$ be a Banach space and $\lambda >0$. Suppose that $\psi ,\phi \in \mathcal{S}_\lambda (0,\infty )$ satisfy \eqref{Lp10b},
    being the integral absolutely convergent.
    If $f\in \mathcal{S}_\lambda (0,\infty )\otimes \mathbb{B}$ and $g\in \mathcal{S}_\lambda (0,\infty )\otimes \mathbb{B}^*$,
    then,
    $$\int_0^\infty \langle g(x),f(x)\rangle _{\mathbb{B}^*,\mathbb{B}}dx
        =\int_0^\infty \int_0^\infty  \langle (g\#_\lambda \phi _{(t)} )(x),(f\#_\lambda \psi _{(t)} )(x)\rangle _{\mathbb{B}^*,\mathbb{B}}\frac{dxdt}{t}.$$
\end{Lem}

Let $f\in \mathcal{S}_\lambda (0,\infty )\otimes \mathbb{B}$. Since $\mathcal{S}_\lambda (0,\infty )\otimes \mathbb{B}^*$ is dense in
$L^{p'}((0,\infty ),\mathbb{B}^*)$, according to \cite[Lemma 2.3]{GLY} we have that
$$\|f\|_{L^p((0,\infty ),\mathbb{B})}
    =\sup_{\substack{g\in \mathcal{S}_\lambda (0,\infty )\otimes \mathbb{B}^* \\ \|g\|_{L^{p'}((0,\infty ),\mathbb{B}^*)}\leq 1}}
        \left|\int_0^\infty \langle g(x),f(x)\rangle _{\mathbb{B}^*,\mathbb{B}}dx\right|.$$
By Lemma \ref{Lema1} we choose $\psi, \phi\in \mathcal{S}_\lambda (0,\infty )$ such that \eqref{Lp10b}
holds, being the integral absolutely convergent. Since $\B^*$ is UMD, it was proved in Section~\ref{subsec:2.1}
that the operator $\mathcal{W}_{\phi ,\mathbb{B}^*}^\lambda $ is bounded from $L^{p'}((0,\infty ),\mathbb{B}^*)$
into $L^{p'}((0,\infty ),\gamma (H, \mathbb{B}^*))$. According to Lemma \ref{Lema3} and \cite[Proposition 2.2]{HW} we get,
for every $g\in \mathcal{S}_\lambda (0,\infty )\otimes \mathbb{B}^*$,
\begin{align*}
    \left|\int_0^\infty \langle g(x),f(x)\rangle _{\mathbb{B}^*,\mathbb{B}}dx\right|
        = &\left|\int_0^\infty \int_0^\infty  \langle (g\#_\lambda \phi _{(t)} )(x),(f\#_\lambda \psi _{(t)} )(x)\rangle _{\mathbb{B}^*,\mathbb{B}}\frac{dxdt}{t}\right|\\
        \leq & \int_0^\infty \int_0^\infty  |\langle (g\#_\lambda \phi _{(t)} )(x),(f\#_\lambda \psi _{(t)} )(x)\rangle _{\mathbb{B}^*,\mathbb{B}}|\frac{dxdt}{t}\\
        \leq &\int_0^\infty \|\mathcal{W}_{\phi ,\mathbb{B}^*}^\lambda (g)(\cdot, x)\|_{\gamma (H,\mathbb{B}^*)}\|\mathcal{W}_{\psi ,\mathbb{B}}^\lambda (f)(\cdot, x)\|_{\gamma (H,\mathbb{B})}dx\\
        \leq &\|\mathcal{W}_{\phi ,\mathbb{B}^*}^\lambda (g)\|_{L^{p'}((0,\infty ), \gamma (H,\mathbb{B}^*))}\|\mathcal{W}_{\psi ,\mathbb{B}}^\lambda (f)\|_{L^p((0,\infty ), \gamma (H,\mathbb{B}))}\\
        \leq &C\|g\|_{L^{p'}((0,\infty ),\mathbb{B}^*)}\|\mathcal{W}_{\psi ,\mathbb{B}}^\lambda (f)\|_{L^p((0,\infty ), \gamma (H,\mathbb{B}))}.
\end{align*}
Hence,
$$\|f\|_{L^p((0,\infty ),\mathbb{B})}
    \leq C\|\mathcal{W}_{\psi ,\mathbb{B}}^\lambda (f)\|_{L^p((0,\infty ), \gamma (H,\mathbb{B}))}.$$
Thus the proof of Theorem \ref{ThBHankel} is finished.

\section{Proof of Theorem~\ref{boundedness}}\label{sec:Proof2}

\subsection{}\label{subsec:3.1}

In this section we prove that
\begin{equation}\label{16.1}
    \| G_{P,\B}^{\lambda,\beta}(f) \|_{L^p((0,\infty),\gamma(H,\B))}
        \leq C \|f\|_{L^p((0,\infty),\B)}, \quad f \in L^p((0,\infty),\B),
\end{equation}
for some $C>0$ independent of $f$.

We define the $g$-operator associated with the classical Poisson semigroup on $\mathbb{R}$ as follows
$$G_{P,\mathbb{B}}^\beta(f)(t,x)
    =t^\beta \partial _t^\beta \PP_t(f)(x),\quad x\in \mathbb{R} \mbox{ and } t\in (0,\infty ),$$
for every $f\in L^p(\mathbb{R},\mathbb{B})$.

By \cite[Proposition 1]{BCCFR1} there exists $C>0$ such that
$$\|G_{P,\mathbb{B}}^\beta (f)\|_{L^p(\mathbb{R}, \gamma (H,\mathbb{B}))}
    \leq C \|f\|_{L^p(\mathbb{R},\mathbb{B})},\quad f\in \mathcal{S} (\mathbb{R})\otimes \mathbb{B}.$$
The arguments developed in the proof of Theorem \ref{ThBHankel} allow us to show that
\begin{equation}\label{Lp12}
    \|G_{P,\mathbb{B}}^\beta (f)\|_{L^p(\mathbb{R}, \gamma (H,\mathbb{B}))}\leq \|f\|_{L^p(\mathbb{R},\mathbb{B})},\quad f\in L^p (\mathbb{R}, \mathbb{B}).
\end{equation}

In \cite[Lemma 1]{BCCFR1} it was established that
$$t^\beta \partial _t^\beta \PP_t(z)
    =\sum_{k=0}^{(m+1)/2}\frac{c_k}{t}\varphi ^k \left(\frac{z}{t}\right),\quad z\in \mathbb{R} \mbox{ and } t\in (0,\infty ),$$
where $m\in \mathbb{N}$ is such that $m-1\leq \beta <m$, and, for every $k\in \mathbb{N}$, $0\leq k\leq (m+1)/2$, $c_k\in \mathbb{C}$ and
$$\varphi ^k(z)
    =\int_0^\infty \frac{(1+v)^{m+1-2k}v^{m-\beta -1}}{((1+v)^2+z^2)^{m-k+1}}dv,\quad z\in \mathbb{R}.$$
By proceeding as in the proof of \cite[Lemma 1]{BCCFR1} we can obtain the analogous identity in the Bessel setting
\begin{equation}\label{16.2}
    t^\beta \partial _t^\beta P_t^\lambda (x,y)
        =\sum_{k=0}^{(m+1)/2}\frac{b_k^\lambda }{t^{2\lambda +1}}(xy)^\lambda \int_0^\pi (\sin \theta )^{2\lambda -1}\varphi^{\lambda ,k}\left(\frac{\sqrt{(x-y)^2+2xy(1-\cos \theta )}}{t}\right)d\theta ,
\end{equation}
where $m\in \mathbb{N}$ is such that $m-1\leq \beta <m$, and, for every $k\in \mathbb{N}$, $0\leq k\leq (m+1)/2$,
$$\varphi ^{\lambda ,k}(z)
    =\int_0^\infty \frac{(1+v)^{m+1-2k}v^{m-\beta -1}}{((1+v)^2+z^2)^{\lambda +m-k+1}}dv,\quad z\in (0,\infty ),$$
and
$$b_k^\lambda = \frac{2 \lambda(\lambda+1) \cdots (\lambda+m-k)}{(m-k)!} c_k.$$

Let $k\in \mathbb{N}$, $0\leq k\leq (m+1)/2$. We define, for every $f\in L^p(\mathbb{R},\mathbb{B})$,
$$\PP_{k}(f)(t,x)
    =\int_\mathbb{R}\varphi _t^k(x-y)f(y)dy,\quad t\in (0,\infty )\mbox{ and }x\in \mathbb{R}.$$
Let $f\in L^p((0,\infty),\mathbb{B})$. If $f_o$ denotes the odd extension of $f$ to $\mathbb{R}$, we write
\begin{align*}
    \PP_{k}(f_o)(t,x)
        = &\int_0^\infty \varphi _t^k(x-y)f(y)dy-\int_0^\infty \varphi _t^k(x+y)f(y)dy\\
        = & \PP_{k,1}(f)(t,x)- \PP_{k,2}(f)(t,x),\quad t,x\in (0,\infty ).
\end{align*}
We have that
\begin{align}\label{Lp13}
    & \|\varphi_t ^k(x+y)\|_H
        \leq \int_0^\infty (1+v)^{m+1-2k}v^{m-\beta -1} \nonumber \\
    & \qquad \times \left(\int_0^\infty \frac{1}{t^3}\frac{t^{4(m-k+1)}}{((1+v^2)t^2+(x+y)^2)^{2(m-k+1)}}dt\right)^{1/2}dv
        \leq  \frac{C}{x+y},\quad x,y\in (0,\infty ).
\end{align}
Since $\gamma (H,\mathbb{C})=H$,  we deduce that
\begin{align*}
    \|\PP_{k,2}(f)(\cdot,x)\|_{\gamma (H,\mathbb{B})}
        \leq &\int_0^\infty \|f(y)\|_\mathbb{B}\left\|\frac{1}{t}\varphi ^k\Big(\frac{x+y}{t}\Big)\right\|_Hdt
        \leq C\int_0^\infty \frac{\|f(y)\|_\mathbb{B}}{x+y}dy\\
        \leq & C  \left[H_0(\|f\|_\mathbb{B})(x)+H_\infty (\|f\|_\mathbb{B})(x)\right],\quad x\in (0,\infty ).
\end{align*}
Thus, according to \cite[p. 244, (9.9.1) and (9.9.2)]{HLP}, $\PP_{k,2}$ is a bounded operator from
$L^p((0,\infty ),\mathbb{B})$ into $L^p((0,\infty ),\gamma (H,\mathbb{B}))$.

We define, for every $f\in L^p((0,\infty ),\mathbb{B})$,
$$G_{P,\mathbb{B}}^{\beta,-}(f)(t,x)
    =\int_0^\infty t^\beta \partial _t^\beta \PP_t(x+y)f(y)dy,\quad t,x\in (0,\infty ).$$
Since $G_{P,\mathbb{B}}^{\beta, -}=\sum_{k=0}^{(m+1)/2}c_k \PP_{k,2}$,
we conclude that  $G_{P,\mathbb{B}}^{\beta, -}$ is a bounded operator from
$L^p((0,\infty ),\mathbb{B})$ into $L^p((0,\infty ),\gamma (H,\mathbb{B}))$.
Then, according to \eqref{Lp12}, if for every $f\in L^p((0,\infty ),\mathbb{B})$, we define
$$G_{P,\mathbb{B}}^{\beta,+}(f)(t,x)
    =\int_0^\infty t^\beta \partial _t^\beta \PP_t(x-y)f(y)dy,\quad t,x\in (0,\infty ),$$
the operator $G_{P,\mathbb{B}}^{\beta ,+}$ is also bounded from $L^p((0,\infty ),\mathbb{B})$ into $L^p((0,\infty ),\gamma (H,\mathbb{B}))$.

In order to prove \eqref{16.1} it is
enough to show that the difference $G_{P,\mathbb{B}}^{\lambda ,\beta}-G_{P,\mathbb{B}}^{\beta ,+}$
is bounded from $L^p((0,\infty ),\mathbb{B})$ into $L^p((0,\infty ),\gamma (H,\mathbb{B}))$.

By proceeding as in \eqref{Lp13} we get, for every $x \in (0,\infty)$,
\begin{equation}\label{Lp14}
    \|t^\beta \partial _t^\beta \PP_t(x-y)\|_H\leq \frac{C}{|x-y|}\leq C\left\{\begin{array}{ll}
                    \displaystyle \frac{1}{x},&\displaystyle 0<y<\frac{x}{2},\\
                    &\\
                    \displaystyle \frac{1}{y},&\displaystyle y>2x.
                    \end{array}
    \right.
\end{equation}
We split $P_t^\lambda (x,y)$, $t,x,y\in (0,\infty )$, as follows
\begin{align*}
    P_t^\lambda (x,y)
        =&\frac{2\lambda (xy)^\lambda t}{\pi }\int_0^{\pi /2}\frac{(\sin \theta )^{2\lambda -1}}{((x-y)^2+t^2+2xy(1-\cos \theta ))^{\lambda +1}}d\theta \\
         & +\frac{2\lambda (xy)^\lambda t}{\pi }\int_{\pi /2}^\pi \frac{(\sin \theta )^{2\lambda -1}}{((x-y)^2+t^2+2xy(1-\cos \theta ))^{\lambda +1}}d\theta\\
        =&P_t^{\lambda ,1}(x,y)+P_t^{\lambda ,2}(x,y).
\end{align*}
From \eqref{16.2} we have that
$$\|t^\beta \partial _t^\beta P_t^{\lambda ,2}(x,y)\|_H
    \leq C(xy)^\lambda \int_{\pi /2}^\pi (\sin \theta )^{2\lambda -1}\sum_{k=0}^{(m+1)/2}\left\|\frac{1}{t^{2\lambda +1}}\varphi ^{\lambda ,k}\left(\frac{\sqrt{(x-y)^2+2xy(1-\cos \theta )}}{t}\right)\right\|_Hd\theta ,$$
and, for every $k\in \mathbb{N}$, $0\leq k\leq (m+1)/2$,
\begin{align*}
    & \left\|\frac{1}{t^{2\lambda +1}}\varphi ^{\lambda ,k}\left(\frac{\sqrt{(x-y)^2+2xy(1-\cos \theta )}}{t}\right)\right\|_H \\
    & \qquad \leq C\int_0^\infty (1+v)^{m+1-2k}v^{m-\beta -1}\left(\int_0^\infty \frac{t^{4(\lambda +1+m-k)-4\lambda -3}}{((1+v^2)t^2+(x+y)^2)^{2(\lambda +m-k+1)}}dt\right)^{1/2}dv\\
    & \qquad \leq \frac{C}{(x+y)^{2\lambda +1}}\int_0^\infty \frac{v^{m-\beta -1}}{(1+v)^m}dv\left(\int_0^\infty \frac{u^{4(m-k)+1}}{(1+u)^{4(\lambda +m-k+1)}}du\right)^{1/2} \\
    & \qquad \leq \frac{C}{(x+y)^{2\lambda +1}},\quad x,y\in (0,\infty )\mbox{ and }\theta \in (\pi/2,\pi ).
\end{align*}
Hence,
\begin{equation}\label{Lp15}
    \|t^\beta \partial _t^\beta P_t^{\lambda ,2}(x,y)\|_H\leq C\frac{(xy)^\lambda}{(x+y)^{2\lambda +1}}\leq \frac{C}{x+y},\quad x,y\in (0,\infty ).
\end{equation}
Similar manipulations lead to
\begin{equation}\label{Lp16}
    \|t^\beta \partial _t^\beta P_t^{\lambda ,1}(x,y)\|_H\leq \frac{C}{|x-y|}\leq C\left\{\begin{array}{ll}
                    \displaystyle \frac{1}{x},&\displaystyle 0<y<x/2,\\
                    &\\
                    \displaystyle \frac{1}{y},&\displaystyle y>2x>0.\\
                    \end{array}
    \right.
\end{equation}
 We decompose $t^\beta \partial _t^\beta P_t^{\lambda ,1}(x,y)$ as follows:
\begin{align*}
    & t^\beta \partial _t^\beta P_t^{\lambda ,1}(x,y)
        =\sum_{k=0}^{(m+1)/2}\frac{b_k^\lambda }{t^{2\lambda +1}}(xy)^\lambda \left\{\int_0^{\pi /2}[(\sin \theta)^{2\lambda -1}-\theta ^{2\lambda -1}]\varphi ^{\lambda ,k}\left(\frac{\sqrt{(x-y)^2+2xy(1-\cos \theta )}}{t}\right)d\theta \right.\\
    & \quad \quad + \int_0^{\pi /2}\theta ^{2\lambda -1}\left[\varphi ^{\lambda ,k}\left(\frac{\sqrt{(x-y)^2+2xy(1-\cos \theta )}}{t}\right)- \varphi ^{\lambda ,k}\left(\frac{\sqrt{(x-y)^2+xy\theta ^2}}{t}\right)\right]d\theta \\
    & \quad \quad -\left. \int_{\pi /2}^\infty \theta ^{2\lambda -1}\varphi ^{\lambda ,k}\left(\frac{\sqrt{(x-y)^2+xy\theta ^2}}{t}\right)d\theta +\int_0^\infty \theta ^{2\lambda -1}\varphi ^{\lambda ,k}\left(\frac{\sqrt{(x-y)^2+xy\theta ^2}}{t}\right)d\theta \right\}\\
    & \quad =\sum_{k=0}^{(m+1)/2}b_k^\lambda [R_1^{\lambda ,k}(t,x,y)+R_2^{\lambda ,k}(t,x,y)+R_3^{\lambda ,k}(t,x,y)+R_4^{\lambda ,k}(t,x,y)],\quad t,x,y\in (0,\infty ).
\end{align*}
Let $k \in \N$, $0 \leq k \leq (m+1)/2$.
By using the mean value theorem we obtain
\begin{align}\label{Lp17}
    \|R_1^{\lambda ,k}(\cdot,x,y)\|_H
        \leq & C(xy)^\lambda \int_0^{\pi /2}\theta ^{2\lambda +1}\int_0^\infty (1+v)^{m+1-2k}v^{m-\beta -1}\nonumber\\
        &\times \left(\int_0^\infty \frac{t^{4(\lambda +1+m-k)-4\lambda-3}}{((1+v)^2t^2+(x-y)^2+xy\theta ^2)^{2(\lambda +m-k+1)}}dt\right)^{1/2}dvd\theta \nonumber\\
        \leq & C(xy)^\lambda \int_0^{\pi /2}\frac{\theta ^{2\lambda +1}}{((x-y)^2+xy\theta ^2)^{\lambda +1/2}}d\theta \leq \frac{C}{x},\quad 0<\frac{x}{2}<y<2x,
\end{align}
and
\begin{align}\label{Lp18}
    & \|R_2^{\lambda ,k}(\cdot,x,y)\|_H
        \leq C(xy)^\lambda \int_0^{\pi /2}\theta ^{2\lambda -1}\int_0^\infty (1+v)^{m+1-2k}v^{m-\beta -1} \Big\{\int_0^\infty t^{4(m-k)+1}\nonumber\\
    & \qquad \qquad \times  \Big|\frac{1}{((1+v)^2t^2+(x-y)^2+2xy(1-\cos \theta ))^{\lambda +m-k+1}} \nonumber\\
    & \qquad \qquad - \frac{1}{((1+v)^2t^2+(x-y)^2+xy\theta ^2)^{\lambda +m-k+1}}\Big|^2dt \Big\}^{1/2}dvd\theta \nonumber\\
    & \qquad \leq  C(xy)^\lambda  \int_0^{\pi /2}\theta ^{2\lambda -1}\int_0^\infty (1+v)^{m+1-2k}v^{m-\beta -1} \nonumber\\
    & \qquad \qquad \times \left\{\int_0^\infty t^{4(m-k)+1}\left(\frac{xy\theta ^4}{((1+v)^2t^2+(x-y)^2+xy\theta ^2)^{\lambda +m-k+2}}\right)^2dt\right\}^{1/2}dvd\theta \nonumber\\
    & \qquad \leq  C(xy)^{\lambda+1} \int_0^{\pi /2}\frac{\theta ^{2\lambda +3}}{((x-y)^2+xy\theta ^2)^{\lambda +3/2}}d\theta
        \leq  \frac{C}{x},\quad 0<\frac{x}{2}<y<2x.
\end{align}
We have also that
\begin{align}\label{Lp19}
    \|R_3^{\lambda ,k}(\cdot,x,y)\|_H
        \leq & C(xy)^\lambda \int_{\pi /2}^\infty \theta ^{2\lambda -1}\int_0^\infty (1+v)^{m+1-2k}v^{m-\beta -1}\nonumber\\
             &\times \left(\int_0^\infty \frac{t^{4(m-k)+1}}{((1+v)^2t^2+(x-y)^2+xy\theta ^2)^{2(\lambda +m-k+1)}}dt\right)^{1/2}dvd\theta \nonumber\\
        \leq&C(xy)^\lambda \int_{\pi /2}^\infty \frac{\theta ^{2\lambda -1}}{((x-y)^2+xy\theta ^2)^{\lambda +1/2}}d\theta
        \leq \frac{C}{x},\quad 0<\frac{x}{2}<y<2x.
\end{align}
Finally, we get that
\begin{align*}
    & \int_0^\infty \theta^{2\lambda -1}\varphi ^{\lambda ,k}\left(\frac{\sqrt{(x-y)^2+xy\theta ^2}}{t}\right)d\theta \\
    & \qquad = \int_0^\infty \theta ^{2\lambda -1}\int_0^\infty \frac{(1+v)^{m+1-2k}v^{m-\beta -1}}{((1+v)^2+[(x-y)^2+xy \theta^2]/t^2)^{\lambda +m-k+1}}dvd\theta \\
    & \qquad = t^{2(\lambda +m-k+1)}\int_0^\infty (1+v)^{m+1-2k}v^{m-\beta -1}\int_0^\infty \frac{\theta ^{2\lambda -1}}{((1+v)^2t^2+(x-y)^2+xy\theta ^2)^{\lambda +m-k+1}}d\theta dv\\
    & \qquad = \frac{t^{2(\lambda +m-k+1)}}{(xy)^{\lambda }}\int_0^\infty \frac{(1+v)^{m+1-2k}v^{m-\beta -1}}{((1+v)^2t^2+(x-y)^2)^{m-k+1}}dv \int_0^\infty \frac{u^{2\lambda -1}}{(1+u^2)^{\lambda +m-k+1}}du\\
    & \qquad = \frac{(m-k)!}{2 \lambda (\lambda+1) \cdots (\lambda+m-k)}\frac{t^{2\lambda}}{(xy)^\lambda}\varphi ^k\Big(\frac{x-y}{t}\Big).
\end{align*}
Then,
\begin{equation}\label{Lp20}
    \sum_{k=0}^{(m+1)/2}b_k^\lambda R_4^{\lambda ,k}(t,x,y)
        =t^\beta \partial _t^ \beta \PP_t(x-y),\quad t,x,y\in (0,\infty ).
\end{equation}
By putting together \eqref{Lp14}-\eqref{Lp20} we conclude that $G_{P,\mathbb{B}}^{\lambda ,\beta}-G_{P,\mathbb{B}}^{\beta ,+}$
is bounded from $L^p((0,\infty ),\mathbb{B})$ into $L^p((0,\infty ),\gamma (H,\mathbb{B}))$, and hence
$G_{P,\mathbb{B}}^{\lambda ,\beta}$ is bounded from $L^p((0,\infty ),\mathbb{B})$ into $L^p((0,\infty ),\gamma (H,\mathbb{B}))$.

\subsection{}\label{subsec:3.2}

We are going to show that there exists $C>0$ such that, for every $f\in L^p((0,\infty ),\mathbb{B})$,
\begin{equation}\label{Lp21}
    \|f\|_{L^p((0,\infty ),\mathbb{B})}\leq C\|G_{P,\mathbb{B}}^{\lambda ,\beta }(f)\|_{L^p((0,\infty ),\gamma (H,\mathbb{B}))}.
\end{equation}
Since $G_{P,\mathbb{B}}^{\lambda ,\beta}$ is bounded from $L^p((0,\infty ),\mathbb{B})$ into
$L^p((0,\infty ),\gamma (H,\mathbb{B}))$ and $\mathcal{S}_\lambda (0,\infty )\otimes \mathbb{B}$ is
a dense subspace of $L^p((0,\infty ),\mathbb{B})$, \eqref{Lp21} holds for every $f\in L^p((0,\infty ),\mathbb{B})$
whenever it is true for every $f\in \mathcal{S}_\lambda (0,\infty )\otimes \mathbb{B}$.

By proceeding as in Section~\ref{subsec:2.2}, the inequality in \eqref{Lp21} can be proved as a
consequence of a polarization identity involving the operator $G_{P,\mathbb{B}}^{\lambda ,\beta}$.
To show this equality we need previously to establish the following.

\begin{Lem}\label{Lem4}
    Let $\lambda$, $\beta>0$. Then, for every $f \in \mathcal{S}_\lambda(0,\infty)$,
    $$h_\lambda\left( t^\beta \partial_t^\beta P_t^\lambda f \right)(x)
        = e^{i\pi \beta }(tx)^\beta e^{-xt} h_\lambda(f)(x), \quad t,x \in (0,\infty).$$
\end{Lem}

\begin{proof}
    Let $f \in \mathcal{S}_\lambda(0,\infty)$. We have that (see \cite[\S 8.5 (19)]{EMOT2})
    $$h_\lambda(P_t^\lambda f)(x)
        = e^{-xt} h_\lambda(f)(x), \quad t,x \in (0,\infty).$$
    We choose $m \in \N$ such that $m-1 \leq \beta < m$. It is not hard to see that
    $\partial_t^\beta e^{-xt}=e^{i\pi \beta}x^\beta e^{-xt}$, $t,x \in (0,\infty)$. Then,
    $$\partial_t^\beta h_\lambda \left( P_t^\lambda f \right)(x)
        = e^{i\pi \beta }x^\beta e^{-xt} h_\lambda(f)(x), \quad t,x \in (0,\infty).$$
    According to \cite[(4.6)]{GLLNU} we can write, for every $t,x,y \in (0,\infty)$ and $\theta \in (0,\pi)$,
    \begin{align*}
        & \partial_t^m \left[ \frac{t}{\left[ (x-y)^2 + 2xy(1-\cos \theta) + t^2\right]^{\lambda+1}} \right]
            = - \frac{1}{2\lambda} \partial_t^{m+1} \left[ \frac{1}{\left[ (x-y)^2 + 2xy(1-\cos \theta) + t^2\right]^{\lambda}} \right] \\
        & \qquad = \frac{1}{2}\sum_{k =0}^{(m+1)/2} (-1)^{m-k} E_{m+1,k} t^{m+1-2k}
                        \frac{(\lambda+1)(\lambda+2) \cdot \dots \cdot (\lambda+m-k)}{\left[ (x-y)^2 + 2xy(1-\cos \theta) + t^2\right]^{\lambda+m-k+1}},
    \end{align*}
    where
    $$E_{m+1,k}
        = \frac{2^{m+1-2k}(m+1)!}{k! (m+1-2k)!}, \quad 0 \leq k \leq \frac{m+1}{2}.$$
    Hence, $\partial_t^m \left[ t/\left[ (x-y)^2 + 2xy(1-\cos \theta) + t^2\right]^{\lambda+1}\right]$ is continuous in
    $(t,x,y,\theta) \in (0,\infty)^3 \times (0,\pi)$. Moreover, for each $t,x,y \in (0,\infty)$ and $\theta \in (0,\pi)$,
    \begin{align*}
        & \left|\partial_t^m \left[ \frac{t}{\left[ (x-y)^2 + 2xy(1-\cos \theta) + t^2\right]^{\lambda+1}} \right] \right|
            \leq \frac{C}{\left[(x-y)^2 +t^2\right]^{\lambda+(m+1)/2}}.
    \end{align*}
    Then,
    $$\left| \partial_t^m P_{t+s}^\lambda(f)(x) \right|
        \leq C \int_0^\infty |f(y) |\frac{(xy)^\lambda}{\left[(x-y)^2 +(t+s)^2\right]^{\lambda+(m+1)/2}} dy, \quad t,x \in (0,\infty),$$
    and $\partial_t^\beta P_t^\lambda (f) \in L^1(0,\infty)$, $t>0$.
    Since the function $\sqrt{z} J_\nu(z)$ is bounded on $(0,\infty)$
    when $\nu>-1/2$, the derivation under the integral sing is justified and we get
    $$h_\lambda\left( \partial_t^\beta P_t^\lambda f \right)(x)
        = \partial_t^\beta h_\lambda\left(P_t^\lambda(f)\right)(x)
        = e^{i\pi \beta }x^\beta e^{-xt} h_\lambda(f)(x), \quad t,x \in (0,\infty).$$
\end{proof}

\begin{Lem}\label{Lem5}
    Let $\B$ be a UMD Banach space and $\lambda$, $\beta>0$. If $f \in \mathcal{S}_\lambda(0,\infty) \otimes \B$ and
    $g \in \mathcal{S}_\lambda(0,\infty) \otimes \B^*$, then
    \begin{equation}\label{Lp22}
        \int_0^\infty \langle g(x), f(x)\rangle_{\B^*,\B} dx
            = \frac{e^{i2\pi \beta}2^{2\beta}}{\G(2\beta)} \int_0^\infty \int_0^\infty \langle t^\beta \partial_t^\beta P_t^\lambda(g)(x) , t^\beta \partial_t^\beta P_t^\lambda(f)(x) \rangle_{\B^*,\B} \frac{dtdx}{t}.
    \end{equation}
\end{Lem}

\begin{proof}
    It is enough to show \eqref{Lp22} when $f,g \in \mathcal{S}_\lambda(0,\infty)$ and $\B=\C$. Let $f,g \in \mathcal{S}_\lambda(0,\infty)$.

    Since $\mathbb{C}$ is a UMD Banach space, as it was proved in Section~\ref{subsec:3.1}, the operator $G_{P,\mathbb{C}}^{\lambda,\beta}$
    is bounded from $L^p(0,\infty)$ into $L^p((0,\infty),H)$, $1<p<\infty$. Hence, the integral in the right hand side of \eqref{Lp22}
    is absolutely convergent.

    As $h_\lambda$ is an isometry in $L^2(0,\infty)$ (\cite[p. 214 and Theorem 129]{Ti}), Lemma~\ref{Lem4} implies
    that $t^\beta \partial_t^\beta P_t^\lambda(f) \in L^2(0,\infty)$ and $t^\beta \partial_t^\beta P_t^\lambda(g) \in L^2(0,\infty)$,
    for every $t>0$. Plancherel equality for Hankel transforms and Lemma~\ref{Lem4} lead to
    \begin{align*}
        & \int_0^\infty \int_0^\infty  t^\beta \partial_t^\beta P_t^\lambda(f)(x) t^\beta \partial_t^\beta P_t^\lambda(g)(x)\frac{dtdx}{t}
            = \int_0^\infty \int_0^\infty  t^\beta \partial_t^\beta P_t^\lambda(f)(x) t^\beta \partial_t^\beta P_t^\lambda(g)(x)\frac{dxdt}{t} \\
        & \qquad = e^{i2\pi\beta} \int_0^\infty \int_0^\infty  (tx)^{2\beta} e^{-2xt} h_\lambda(f)(x) h_\lambda(g)(x) \frac{dxdt}{t} \\
        & \qquad = e^{i2\pi\beta} \int_0^\infty h_\lambda(f)(x) h_\lambda(g)(x) \int_0^\infty  (tx)^{2\beta} e^{-2xt}  \frac{dtdx}{t} \\
        & \qquad = e^{i2\pi\beta}\frac{\G(2\beta)}{2^{-2\beta}} \int_0^\infty h_\lambda(f)(x) h_\lambda(g)(x)dx
            = e^{i2\pi\beta} \frac{\G(2\beta)}{2^{-2\beta}}  \int_0^\infty f(x) g(x)dx.
    \end{align*}
\end{proof}

By using now Lemma~\ref{Lem5}, the arguments developed in Section~\ref{subsec:2.2}
allow us to show that \eqref{Lp21} holds, for every $f \in L^p((0,\infty),\B)$.

Thus the proof of Theorem~\ref{boundedness} is completed.

\section{Proof of Theorem~\ref{boundedness2}}\label{sec:Proof3}

The Riesz transform $R_\lambda$ associated with the Bessel operator $\Delta_\lambda$ is the principal value integral operator
defined, for every $f \in L^p(0,\infty)$, by
$$R_\lambda(f)(x)
    = \lim_{\varepsilon \to 0^+} \int_{0, \ |x-y|>\varepsilon}^\infty R_\lambda(x,y) f(y)dy, \quad \text{a.e.  } x \in (0,\infty),$$
where
$$R_\lambda(x,y)
    = \int_0^\infty D_\lambda P_t^\lambda(x,y)dt, \quad x,y \in (0,\infty), \ x \neq y,$$
and $D_\lambda=x^\lambda \frac{d}{dx} x^{-\lambda}$. Main properties of Riesz transform $R_\lambda$ can be encountered in \cite{BBFMT}.
We denote by $R_\lambda^*$ the "adjoint" operator of $R_\lambda$ defined, for every $f \in L^p(0,\infty)$, by
$$R_\lambda^*(f)(x)
    = \lim_{\varepsilon \to 0^+} \int_{0, \ |x-y|>\varepsilon}^\infty R_\lambda(y,x) f(y)dy, \quad \text{a.e.  } x \in (0,\infty).$$
Riesz transforms $R_\lambda$ and $R_\lambda^*$ are bounded from $L^p(0,\infty)$ into itself. Moreover, since $\B$ is a UMD Banach space,
by defining $R_\lambda$ and $R_\lambda^*$ on $L^p(0,\infty)\otimes \B$ in the natural way, $R_\lambda$ and $R_\lambda^*$ can be extended to
$L^p((0,\infty),\B)$ as bounded operators on  $L^p((0,\infty),\B)$ into itself (\cite[Theorem 2.1]{BFMT}).

We define, for every $f \in L^p(0,\infty)$, the function $\QQ_t^\lambda(f)$ by
$$\QQ_t^\lambda(f)(x)
    = \int_0^\infty \QQ_t^\lambda(x,y)f(y)dy, \quad t,x \in (0,\infty),$$
where
$$\QQ_t^\lambda(x,y)
    = \frac{2\lambda(xy)^\lambda}{\pi} \int_0^\pi \frac{(x-y\cos \theta)(\sin \theta)^{2\lambda-1}}{(x^2+y^2+t^2-2xy \cos \theta)^{\lambda+1}} d\theta, \quad t,x,y \in (0,\infty).$$
The following Cauchy-Riemann equations hold
$$D_\lambda P_t^\lambda(f)= \partial_t \QQ_t^\lambda(f), \qquad
  D_\lambda^* \QQ_t^\lambda(f) = \partial_t P_t^\lambda(f), \quad t>0.$$
These relations motivate that $\QQ_t^\lambda(f)$ is called $\Delta_\lambda$-conjugated to the Poisson integral $P_t^\lambda(f)$.

The adjoint $\Delta_\lambda$-conjugated $\Q_t^\lambda(f)$ of $f \in L^p(0,\infty)$ is defined by
$$\Q_t^\lambda(f)(x)
    = \int_0^\infty \QQ_t^\lambda(y,x)f(y)dy, \quad t,x \in (0,\infty).$$
We have that
$$D_\lambda^* P_t^{\lambda+1}(f)= \partial_t \Q_t^\lambda(f), \qquad
  D_\lambda \Q_t^\lambda(f) = \partial_t P_t^{\lambda+1}(f), \quad t>0.$$
By using Hankel transform (see \cite[(16.5)]{MS}) we can see that, for every $f \in \mathcal{S}_\lambda(0,\infty)$,
$$P_t^\lambda(R_\lambda^* f)
    = \Q_t^\lambda(f), \quad t>0.$$
Then, for every $f \in \mathcal{S}_\lambda(0,\infty)$,
\begin{equation}\label{Lp23}
    \partial_t P_t^\lambda(R_\lambda^* f)
        = D_\lambda^* P_t^{\lambda+1}(f), \quad t>0.
\end{equation}
Equality \eqref{Lp23} also holds for every $f \in \mathcal{S}_\lambda(0,\infty) \otimes \B$. Then,
\begin{equation}\label{Lp24}
    \mathcal{G}_{P,\B}^{\lambda}(f)
        = G_{P,\B}^{\lambda,1}(R_\lambda^* f), \quad f \in \mathcal{S}_\lambda(0,\infty) \otimes \B.
\end{equation}

Since $R_\lambda^*$ can be extended to $L^p((0,\infty),\B)$ as a bounded operator from $L^p((0,\infty),\B)$
into itself, Theorem~\ref{boundedness} implies that the operator $\mathcal{G}_{P,\B}^{\lambda}$
can be extended from $\mathcal{S}_\lambda(0,\infty) \otimes \B$ as a bounded operator from $L^p((0,\infty),\B)$ into
$L^p((0,\infty),\gamma(H,\B))$.
We denote this extension by $\widetilde{\mathcal{G}}_{P,\B}^{\lambda}$.

We define
$$\mathcal{G}_{P,\B}^{\lambda}(t,x,y)
    = t D_\lambda^* P_t^{\lambda+1}(x,y), \quad t,x,y \in (0,\infty).$$
We have that
\begin{align*}
    \mathcal{G}_{P,\B}^{\lambda}(t,x,y)
        = & - \frac{2(\lambda+1)}{\pi} t^2 y^{\lambda+1} x^{-\lambda}
                \partial_x \left( x^{2\lambda+1} \int_0^\pi \frac{(\sin \theta)^{2\lambda+1}}{\left[ (x-y)^2 + t^2 + 2xy(1-\cos \theta) \right]^{\lambda+2}} d\theta \right) \\
        = & - \frac{2(\lambda+1)(2\lambda+1)}{\pi} t^2 x^{\lambda} y^{\lambda+1}
                \int_0^\pi \frac{(\sin \theta)^{2\lambda+1}}{\left[ (x-y)^2 + t^2 + 2xy(1-\cos \theta) \right]^{\lambda+2}} d\theta \\
          & + \frac{4(\lambda+1)(\lambda+2)}{\pi} t^2 (xy)^{\lambda+1}
                \int_0^\pi \frac{[(x-y)+y(1-\cos \theta)](\sin \theta)^{2\lambda+1}}{\left[ (x-y)^2 + t^2 + 2xy(1-\cos \theta) \right]^{\lambda+3}} d\theta,
          \quad t,x,y \in (0,\infty).
\end{align*}
Then,
\begin{align*}
    \left| \mathcal{G}_{P,\B}^{\lambda}(t,x,y) \right|
        \leq &  C \sqrt{t} \Big\{ x^\lambda y^{\lambda+1} \left(\int_0^{\pi/2} + \int_{\pi/2}^\pi \right) \frac{(\sin \theta)^{2\lambda+1}}{\left[ (x-y)^2 + t^2 + 2xy(1-\cos \theta) \right]^{\lambda+5/4}} d\theta \\
             & + (xy)^{\lambda+1} \left(\int_0^{\pi/2} + \int_{\pi/2}^\pi \right) \frac{(\sin \theta)^{2\lambda+1}}{\left[ (x-y)^2 + t^2 + 2xy(1-\cos \theta) \right]^{\lambda+7/4}} d\theta \Big\} \\
           = & \mathcal{G}_{P,\B}^{\lambda,1,1}(t,x,y) + \mathcal{G}_{P,\B}^{\lambda,1,2}(t,x,y)
                + \mathcal{G}_{P,\B}^{\lambda,2,1}(t,x,y) + \mathcal{G}_{P,\B}^{\lambda,2,2}(t,x,y) ,\quad t,x,y \in (0,\infty).
\end{align*}

Let $\varepsilon >0$. Since, for every $x,y \in (0,\infty)$ and $\theta \in (0,\pi/2)$,
$$\left( \int_\varepsilon^\infty \frac{dt}{\left[ (x-y)^2 + t^2 + 2xy(1-\cos \theta) \right]^{2\lambda+5/2}} \right)^{1/2}
    \leq \frac{C}{(|x-y|+\varepsilon+\sqrt{xy}\theta)^{2\lambda+2}},$$
and
$$\left( \int_\varepsilon^\infty \frac{dt}{\left[ (x-y)^2 + t^2 + 2xy(1-\cos \theta) \right]^{2\lambda+7/2}} \right)^{1/2}
    \leq \frac{C}{(|x-y|+\varepsilon+\sqrt{xy}\theta)^{2\lambda+3}},$$
we obtain
\begin{align*}
    &\Big\| \mathcal{G}_{P,\B}^{\lambda,1,1}(\cdot,x,y) \Big\|_{L^2\left((\varepsilon,\infty),dt/t\right)}
        + \Big\| \mathcal{G}_{P,\B}^{\lambda,2,1}(\cdot,x,y) \Big\|_{L^2\left((\varepsilon,\infty),dt/t\right)} \\
    & \qquad \leq  C \left( x^\lambda y^{\lambda+1} \int_0^{\pi/2} \frac{\theta^{2\lambda+1}}{(|x-y|+\varepsilon+\sqrt{xy}\theta)^{2\lambda+2}} d\theta
     + (xy)^{\lambda+1} \int_0^{\pi/2} \frac{\theta^{2\lambda+1}}{(|x-y|+\varepsilon+\sqrt{xy}\theta)^{2\lambda+3}} d\theta \right) \\
    & \qquad \leq  C \left( \frac{y}{(|x-y|+\varepsilon)^2} + \frac{xy}{(|x-y|+\varepsilon)^3} \right)
        \leq  C \left\{\begin{array}{ll}
                            1/(x+\varepsilon), & 0<y<x/2, \\
                            y/\varepsilon^2+y^2/\varepsilon^3, & x/2<y<2x, \\
                            1/(y+\varepsilon), & y>2x>0.
                        \end{array} \right.
\end{align*}
Analogously,
\begin{align*}
    &\Big\| \mathcal{G}_{P,\B}^{\lambda,1,2}(\cdot,x,y) \Big\|_{L^2\left((\varepsilon,\infty),dt/t\right)}
        + \Big\| \mathcal{G}_{P,\B}^{\lambda,2,2}(\cdot,x,y) \Big\|_{L^2\left((\varepsilon,\infty),dt/t\right)} \\
    & \qquad \leq  C \left( x^\lambda y^{\lambda+1} \int_{\pi/2}^\pi \frac{(\sin \theta)^{2\lambda+1}}{(x+y+\varepsilon)^{2\lambda+2}} d\theta
     + (xy)^{\lambda+1} \int_{\pi/2}^\pi \frac{(\sin \theta)^{2\lambda+1}}{(x+y+\varepsilon)^{2\lambda+3}} d\theta \right) \\
    & \qquad \leq  \frac{C}{x+y+\varepsilon},\quad x,y \in (0,\infty).
\end{align*}
Hence, for every $x \in (0,\infty)$,
$\left\| \mathcal{G}_{P,\B}^{\lambda}(\cdot,x,y) \right\|_{L^2\left((\varepsilon,\infty),dt/t\right)} \in L^{p'}(0,\infty)$.

By proceeding now as in Section~\ref{subsec:2.1} we conclude that
$$\mathcal{G}_{P,\B}^{\lambda}(f)
    = \widetilde{\mathcal{G}}_{P,\B}^{\lambda}(f), \quad f \in L^p((0,\infty), \B).$$
and the proof of Theorem~\ref{boundedness2} is completed.

\section{Proof of Theorem~\ref{Th1.4}}\label{sec:Proof4}

\subsection{Proof of $(i) \Rightarrow (ii)$ and $(i) \Rightarrow (iii)$ }\label{subsec:4.1}

In Theorems~\ref{boundedness} and \ref{boundedness2} it was proved that if $\B$ is a UMD Banach space,
then \eqref{A2}, \eqref{A3} and \eqref{A1} are satisfied, for every $1<p<\infty$.

\subsection{Proof of $(ii) \Rightarrow (i)$}\label{subsec:4.2}

Let $1<p<\infty$. Suppose that \eqref{A2} and \eqref{A3} hold. Let $f \in \mathcal{S}_\lambda(0,\infty)\otimes \B$. Since $R_\lambda^*$
is bounded from $L^p(0,\infty)$ into itself (\cite[Theorem 4.2]{BBFMT}),
$R_\lambda^*f \in L^p(0,\infty)\otimes \B$. According to \eqref{Lp24} we obtain
\begin{align*}
    \|R_\lambda^* f\|_{L^p((0,\infty),\B)}
        \leq & C \|G_{P,\B}^{\lambda,1}(R_\lambda^* f)\|_{L^p((0,\infty),\gamma(H,\B))}
        =      C \|\mathcal{G}_{P,\B}^{\lambda}(f)\|_{L^p((0,\infty),\gamma(H,\B))} \\
        \leq & C \| f \|_{L^p((0,\infty),\B).}
\end{align*}
Since $\mathcal{S}_\lambda(0,\infty) \otimes \B$ is dense in $L^p((0,\infty),\B)$, $R_\lambda^*$ can be extended to
$L^p((0,\infty),\B)$ as a bounded operator from $L^p((0,\infty),\B)$ into itself. By using
\cite[Theorem 2.1]{BFMT} we deduce that $\B$ is UMD.

\subsection{Proof of $(iii) \Rightarrow (i)$}\label{subsec:4.3}

Assume now that \eqref{A1} holds. In order to show that $\B$ is UMD we prove previously a characterization of UMD Banach
spaces involving $L^p$-boundedness properties of the imaginary powers $\Delta_\lambda^{i\omega}$,
$\omega \in \mathbb{R} \setminus \{0\}$, of the Bessel operator $\Delta_\lambda$.

Let $\omega \in \mathbb{R} \setminus \{0\}$. The $i \omega$-power $\Delta_\lambda^{i\omega}$ of $\Delta_\lambda$
is the Hankel multiplier defined by
\begin{equation}\label{23.1}
    \Delta_\lambda^{i\omega} f
        = h_\lambda \left( y^{2i \omega} h_\lambda(f)\right), \quad f \in L^2(0,\infty).
\end{equation}
Since $h_\lambda$ is an isometry in $L^2(0,\infty)$, the operator $\Delta_\lambda^{i\omega}$ is bounded from
$L^2(0,\infty)$ into itself. Moreover
$$y^{2i \omega}
    = y^2 \int_0^\infty e^{-y^2 u} \frac{u^{-i \omega}}{\G(1-i \omega)} du, \quad y \in (0,\infty),$$
and hence $\Delta_\lambda^{i\omega}$ is a Hankel multiplier of Laplace transform type. This type of Hankel multipliers
were studied in \cite{BCC1} and \cite{BMR}. Proceeding as in \cite[Theorem 1.2]{BCC1}, for every
$f \in C_c^\infty(0,\infty)$, we have that
\begin{equation}\label{Lp25}
    \Delta_\lambda^{i\omega} f (x)
        = \lim_{\varepsilon \to 0^+} \left( \alpha(\varepsilon)f(x) - \int_{0, \ |x-y|>\varepsilon}^\infty K^\lambda_\omega(x,y)f(y)dy\right),
        \quad \text{a.e.  } x \in (0,\infty),
\end{equation}
where
$$K^\lambda_\omega(x,y)
    = \int_0^\infty \frac{t^{-i\omega}}{\G(1-i \omega)} \partial_t W_t^\lambda(x,y)dt, \quad x,y \in (0,\infty), \ x \neq y,$$
and $W_t^\lambda(x,y)$ is the Bessel heat kernel
$$W_t^\lambda(x,y)
    = \frac{1}{\sqrt{2t}} \left( \frac{xy}{2t}\right)^{1/2} I_{\lambda-1/2}\left( \frac{xy}{2t} \right) e^{-(x^2+y^2)/4t}, \quad t,x,y \in (0,\infty).$$
Here $\alpha$ denotes a bounded function on $(0,\infty)$ and $I_\nu$ is the modified Bessel function of the first kind and order $\nu$.
By \cite[Theorem 1.2]{BMR}, $\Delta_\lambda^{i\omega}f$ can be
extended to $L^p(0,\infty)$ as a bounded operator from  $L^p(0,\infty)$ into itself. Moreover, as in \cite[Theorem 1.4]{BCC1},
we can see that this extension, that we will continue denoting by $\Delta_\lambda^{i\omega}$, is given by the limit in \eqref{Lp25}
for every $f \in L^p(0,\infty)$. The operator $\Delta_\lambda^{i\omega}$, is defined on $L^p(0,\infty) \otimes \B$ in the usual way.

The following result is a Bessel version of \cite[Theorem p. 402]{Gue}.

\begin{Prop}\label{ImagBess}
    Let $X$ be a Banach space and $\lambda>0$. Then, $X$ is UMD if and only if, for some (equivalently, for every) $1<q<\infty$,
    the operator $\Delta^{i\omega}_\lambda$, $\omega \in \mathbb{R}\setminus \{0\}$, can be extended from
    $L^q(0,\infty) \otimes X$ to $L^q((0,\infty),X)$ as a bounded operator from $L^q((0,\infty),X)$ into itself.
\end{Prop}

\begin{proof}
    According to \cite[Theorem p. 402]{Gue}, $X$ is UMD if, and only if, for every $\omega \in \mathbb{R}\setminus\{0\}$  and
    for some (equivalently, for every) $1<q<\infty$, the $i\omega$-power $\left( - \frac{d^2}{dx^2}\right)^{i \omega}$
    of the operator $- \frac{d^2}{dx^2}$ can be extended from $L^q(\mathbb{R}) \otimes X$ to $L^q(\mathbb{R},X)$ as a
    bounded operator from $L^q(\mathbb{R},X)$ into itself.

    We recall that (see \cite[Appendix]{BCFR} for a proof) for every $f \in L^q(\mathbb{R})$, $1<q<\infty$, and
    $\omega \in \mathbb{R}\setminus \{0\}$,
    $$\left( - \frac{d^2}{dx^2}\right)^{i \omega} f(x)
        = \lim_{\varepsilon \to 0^+} \left( \alpha(\varepsilon)f(x) - \int_{\ |x-y|>\varepsilon} K_\omega(x,y)f(y)dy\right),
        \quad \text{a.e.  } x \in \mathbb{R},$$
    where
    $$K_\omega (x,y)
        = -\int_0^\infty \frac{t^{-i\omega}}{\G(1-i \omega)} \partial_t \W_t(x-y)dt, \quad x,y \in \mathbb{R}, \ x \neq y,$$
    and $\W_t(z)$ denotes the classical heat kernel \eqref{2.1}.
    Here $\alpha$ represents the same function that appears in \eqref{Lp25}. The operator
    $\left( - \frac{d^2}{dx^2}\right)^{i \omega}$, $\omega \in \mathbb{R}\setminus\{0\}$, is defined on $L^q(\mathbb{R}) \otimes X$,
    $1<q<\infty$, in the natural way.

    Let $\omega \in \R \setminus \{0\}$. We are going to obtain some estimates for the kernels $K^\lambda_\omega(x,y)$
    and $K_\omega(x,y)$, $x,y \in (0,\infty )$, that will allow us to get our characterization of the UMD spaces by using
    imaginary powers of Bessel operators.

    Note firstly that
    \begin{equation}\label{Lp26}
        |K_\omega(x,-y)|
            \leq C \int_0^\infty \left| \partial_t \W_t(x+y) \right| dt
            \leq C \int_0^\infty \frac{e^{-c(x+y)^2/t}}{t^{3/2}} dt
            \leq \frac{C}{x+y}, \quad x,y \in (0,\infty).
    \end{equation}
    In a similar way we obtain, for every $x \in (0,\infty)$,
    \begin{equation}\label{Lp27}
        |K_\omega(x,y)|
            \leq C \left\{ \begin{array}{ll}
                                1/x, & 0<y<x/2,\\
                                1/y, & y>2x.
                           \end{array} \right.
    \end{equation}
    According to \cite[pp. 108 and 123]{Leb} if $\nu>-1$, we have that
    \begin{equation}\label{Lp28}
        I_\nu(z) \sim \frac{z^\nu}{2^\nu \G(\nu+1)}, \quad \text{as } z \to 0^+,
    \end{equation}
    and
    \begin{equation}\label{Lp29}
        \sqrt{z} I_\nu(z)
            = \frac{e^z}{\sqrt{2\pi}} \left( \sum_{r=0}^n \frac{(-1)^r [\nu,r]}{(2z)^r} + \mathcal{O}\left(\frac{1}{z^{n+1}}\right)\right),
    \end{equation}
    where $[\nu,0]=1$ and
    $$[\nu,r]=\frac{(4\nu^2-1)(4\nu^2-3^2) \cdot \dots \cdot (4\nu^2-(2r-1)^2)}{2^{2r}\G(r+1)}, \quad r=1,2, \dots$$
    Since (see \cite[p. 110]{Leb}),
    \begin{equation}\label{24.1}
        \frac{d}{dz}\left( z^{-\nu} I_\nu(z)\right)
            = z^{-\nu} I_{\nu+1}(z), \quad z \in (0,\infty), \ \nu>-1,
    \end{equation}
    it follows that, for every $t,x,y \in (0,\infty)$,
    \begin{align*}
        & \partial_t \left[ W_t^\lambda(x,y) - \W_t(x-y) \right]
                = \partial_t \left[ \W_t(x-y) \left\{ \sqrt{2\pi} \left( \frac{xy}{2t} \right)^{\nu+1/2} \left( \frac{xy}{2t} \right)^{-\nu} I_\nu\left( \frac{xy}{2t}\right)e^{-xy/2t}-1\right\} \right] \\
        & \qquad = \partial_t \W_t(x-y) \left\{ \sqrt{2\pi} \left( \frac{xy}{2t} \right)^{1/2} I_\nu\left( \frac{xy}{2t}\right)e^{-xy/2t}-1\right\} \\
        & \qquad \qquad - \sqrt{2\pi} \W_t(x-y) \Big\{ (\nu+1/2) \left( \frac{xy}{2t} \right)^{\nu-1/2} \frac{xy}{2t^2} \left( \frac{xy}{2t} \right)^{-\nu} I_\nu\left( \frac{xy}{2t}\right) \\
        & \qquad \qquad + \left( \frac{xy}{2t} \right)^{\nu+1/2} \frac{xy}{2t^2} \left( \frac{xy}{2t} \right)^{-\nu} I_{\nu+1}\left( \frac{xy}{2t}\right)
                        - \frac{xy}{2t^2} \left( \frac{xy}{2t} \right)^{1/2} I_{\nu}\left( \frac{xy}{2t}\right) \Big\} e^{-xy/2t} \\
        & \qquad = \partial_t \W_t(x-y) \left\{ \sqrt{2\pi} \left( \frac{xy}{2t} \right)^{1/2} I_\nu\left( \frac{xy}{2t}\right)e^{-xy/2t}-1\right\} \\
        & \qquad \qquad - \sqrt{2\pi} \W_t(x-y) \frac{xy}{2t^2} e^{-xy/2t} \\
        & \qquad \qquad \times \left\{  (\nu+1/2) \frac{2t}{xy} \left( \frac{xy}{2t} \right)^{1/2} I_{\nu}\left( \frac{xy}{2t}\right)
                                       + \left( \frac{xy}{2t} \right)^{1/2} I_{\nu+1}\left( \frac{xy}{2t}\right)
                                       - \left( \frac{xy}{2t} \right)^{1/2} I_{\nu}\left( \frac{xy}{2t}\right) \right\},
    \end{align*}
    being $\nu=\lambda-1/2$.

    From \eqref{Lp28}, we deduce
    \begin{equation}\label{Lp30}
        \left| \partial_t \left[ W_t^\lambda(x,y)-\W_t(x-y) \right] \right|
            \leq C \frac{e^{-c(x-y)^2/t}}{t^{3/2}}, \quad t,x,y \in (0,\infty) \text{ and } xy \leq 2t,
    \end{equation}
    and by using \eqref{Lp29},
    \begin{equation}\label{Lp31}
        \left| \partial_t \left[ W_t^\lambda(x,y)-\W_t(x-y) \right] \right|
            \leq C \frac{e^{-c(x-y)^2/t}}{t^{1/2}xy}, \quad t,x,y \in (0,\infty) \text{ and } xy \geq 2t.
    \end{equation}
    Combining \eqref{Lp30} and \eqref{Lp31} we obtain
    \begin{align}\label{Lp32}
         \left| K_\omega(x,y)-K_\omega^\lambda(x,y) \right|
            \leq &C \int_0^\infty \left| \partial_t \left[ W_t^\lambda(x,y) - \W_t(x-y) \right] \right| dt \nonumber \\
            \leq& C \left( \int_0^{xy/2} \frac{e^{-c(x-y)^2/t}}{t^{1/2}xy} dt + \int_{xy/2}^\infty \frac{e^{-c(x^2+y^2)/t}}{t^{3/2}} dt \right) \nonumber\\
            \leq& \frac{C}{(xy)^{1/2}} \leq \frac{C}{x}, \quad x/2 < y < 2x, \quad x \in (0,\infty).
    \end{align}
    Moreover, \eqref{Lp30} and \eqref{Lp31} imply that, for each $x \in (0,\infty)$,
    \begin{align}\label{Lp33}
        \left| K_\omega^\lambda(x,y) \right|
            \leq C \int_0^\infty \frac{e^{-c(x-y)^2/t}}{t^{3/2}}dt
            \leq \frac{C}{|x-y|}
            \leq C \left\{
                        \begin{array}{ll}
                            1/x, & 0<y<x/2, \\
                            1/y, & y>2x.
                        \end{array}\right.
    \end{align}
    Suppose that $X$ is UMD and $1<q<\infty$. Let $f \in L^q(0,\infty) \otimes X$. We define the function $\tilde{f}$ by
    $$\tilde{f}(x)
        = \left\{
                \begin{array}{ll}
                    0, & x \leq 0, \\
                    f(x), & x>0.
                \end{array}\right.$$
    Thus, $\tilde{f} \in L^q(\R)\otimes X$.
    We have that
    $$\left( - \frac{d^2}{dx^2} \right)^{i \omega} \tilde{f}(x)
        = \lim_{\varepsilon \to 0^+} \left( f(x)\alpha(\varepsilon) - \int_{0, |x-y|>\varepsilon}^\infty K_\omega(x,y)f(y)dy \right), \quad \text{a.e.  } x \in (0,\infty),$$
    and
    $$\Delta_\lambda^{i \omega} f(x)
        = \lim_{\varepsilon \to 0^+} \left( f(x)\alpha(\varepsilon) - \int_{0, |x-y|>\varepsilon}^\infty K_\omega^\lambda(x,y)f(y)dy \right), \quad \text{a.e.  } x \in (0,\infty).$$
    Then, \eqref{Lp27}, \eqref{Lp32}, \eqref{Lp33} lead to
    \begin{align*}
        & \Big\| \left( - \frac{d^2}{dx^2} \right)^{i \omega} \tilde{f}(x) - \Delta_\lambda^{i \omega} f(x) \Big\|_X
            \leq {\underset{\varepsilon \to 0^+}{\overline{\lim}}} \int_{0, |x-y|>\varepsilon}^\infty \left| K_\omega(x,y) - K_\omega^\lambda(x,y) \right| \|f(y)\|_X dy \\
        & \qquad \leq C \left[ H_0(\|f\|_X)(x) + H_\infty(\|f\|_X)(x)  \right], \quad \text{a.e.  } x \in (0,\infty).
    \end{align*}
    Hence, according to \cite[p. 244, (9.9.1) and (9.9.2)]{HLP}, there exists $C>0$ such that
    $$\Big\| \left( - \frac{d^2}{dx^2} \right)^{i \omega} \tilde{f} - \Delta_\lambda^{i \omega} f \Big\|_{L^q((0,\infty),X)}
        \leq C \|f\|_{L^q((0,\infty),X)}, \quad f \in L^q(0,\infty)\otimes X.$$
    Moreover, by \cite[Theorem p. 402]{Gue}, we also have
    $$\Big\| \left( - \frac{d^2}{dx^2} \right)^{i \omega} \tilde{f} \Big\|_{L^q((0,\infty),X)}
        \leq C \|f\|_{L^q((0,\infty),X)}, \quad f \in L^q(0,\infty)\otimes X.$$
    We conclude that
    $$\Big\| \Delta_\lambda^{i \omega} f \Big\|_{L^q((0,\infty),X)}
        \leq C \|f\|_{L^q((0,\infty),X)}, \quad f \in L^q(0,\infty)\otimes X.$$

    Suppose now that $\Delta_\lambda^{i\omega}$ can be extended from $L^q(0,\infty)\otimes X$ to $L^q((0,\infty),X)$ as a bounded
    operator from $L^q((0,\infty),X)$ into itself. According to \cite[Theorem p. 402]{Gue} in order to see that $X$ is UMD it is sufficient
    to see that, for a certain $C>0$,
    $$\Big\| \left( - \frac{d^2}{dx^2} \right)^{i \omega} f \Big\|_{L^q(\R,X)}
        \leq C \|f\|_{L^q(\R,X)}, \quad f \in L^q(\R)\otimes X.$$
    Let $f \in L^p(\mathbb{R})\otimes X$. By defining
    $$f_+(x)=f(x), \quad \text{and} \quad f_-(x)=f(-x), \quad x \in (0,\infty),$$
    we have that
    \begin{align*}
        & \left( - \frac{d^2}{dx^2} \right)^{i \omega} f(x)
            = \lim_{\varepsilon \to 0^+} \left( f_+(x)\alpha(\varepsilon) - \int_{0, |x-y|>\varepsilon}^\infty K_\omega(x,y)f_+(y)dy
                                               - \int_{-\infty}^0 K_\omega(x,y)f(y)dy  \right) \\
        \qquad & = \lim_{\varepsilon \to 0^+} \left( f_+(x)\alpha(\varepsilon) - \int_{0, |x-y|>\varepsilon}^\infty K_\omega(x,y)f_+(y)dy\right)
                                               - \int_0^\infty K_\omega(x,-y)f_-(y)dy  \, \quad \text{a.e.  } x \in (0,\infty),
    \end{align*}
    and
    \begin{align*}
        & \left( - \frac{d^2}{dx^2} \right)^{i \omega} f(x)
            = \lim_{\varepsilon \to 0^+} \left( f(x)\alpha(\varepsilon) - \int_{-\infty, |x-y|>\varepsilon}^0 K_\omega(x,y)f(y)dy
                                               - \int_0^\infty K_\omega(x,y)f(y)dy  \right) \\
        \qquad & = \lim_{\varepsilon \to 0^+} \left( f_-(-x)\alpha(\varepsilon) - \int_{0, |x+y|>\varepsilon}^\infty K_\omega(x,-y)f(-y)dy
                                               - \int_0^\infty K_\omega(x,y)f(y)dy  \right)\\
        \qquad & = \lim_{\varepsilon \to 0^+} \left( f_-(-x)\alpha(\varepsilon) - \int_{0, |x+y|>\varepsilon}^\infty K_\omega(x,-y)f_-(y)dy\right)
                                               - \int_0^\infty K_\omega(x,y)f_+(y)dy  , \quad \text{a.e.  } x \in (-\infty,0).
    \end{align*}
    We consider the operators
    $$T_{\omega,1}(g)(x)
        = \lim_{\varepsilon \to 0^+} \left( g(x)\alpha(\varepsilon) - \int_{0, |x-y|>\varepsilon}^\infty K_\omega(x,y)g(y)dy\right), \quad x\in (0,\infty),$$
    and
    $$T_{\omega,2}(g)(x)
        = \int_{0}^\infty K_\omega(x,-y)g(y)dy, \quad x\in (0,\infty),$$
    for every $g \in L^q(0,\infty) \otimes X$.

    We can write
    \begin{align}\label{Lp34}
        \Big\| \left( - \frac{d^2}{dx^2} \right)^{i \omega} f\Big\|_{L^q(\R,X)}^q
            = & \Big\| T_{\omega,1}(f_+) \Big\|_{L^q((0,\infty),X)}^q
              + \Big\| T_{\omega,2}(f_-) \Big\|_{L^q((0,\infty),X)}^q \nonumber \\
            & + \Big\| T_{\omega,1}(f_-) \Big\|_{L^q((0,\infty),X)}^q
              + \Big\| T_{\omega,2}(f_+) \Big\|_{L^q((0,\infty),X)}^q
    \end{align}
    According to \eqref{Lp26} we get, for every $g \in L^q(0,\infty) \otimes X$,
    $$\| T_{\omega,2}(g)(x) \|_X
        \leq C \int_0^\infty \frac{\|g(y)\|_X}{x+y}dy
        \leq C \left[ H_0(\|g\|_X)(x) + H_\infty(\|g\|_X)(x)\right], \quad x\in (0,\infty). $$
    Also, by combining \eqref{Lp27}, \eqref{Lp32} and \eqref{Lp33} we obtain, for each $g \in L^q(0,\infty) \otimes X$,
    $$\| T_{\omega,1}(g)(x) - \Delta_\lambda^{i\omega}(g)(x) \|_X
        \leq C \left[ H_0(\|g\|_X)(x) + H_\infty(\|g\|_X)(x)\right], \quad x\in (0,\infty). $$
    Then, by \cite[p. 244, (9.9.1) and (9.9.2)]{HLP} it follows that , for every $g \in L^q(0,\infty) \otimes X$,
    \begin{equation}\label{Lp35}
        \| T_{\omega,2}(g)\|_{L^q((0,\infty),X)} + \| T_{\omega,1}(g) - \Delta_\lambda^{i\omega}(g)\|_{L^q((0,\infty),X)}
            \leq C \|g\|_{L^q((0,\infty),X)}.
    \end{equation}
    Since $\Delta_\lambda^{i\omega}$ can be extended from $L^q(0,\infty) \otimes X$ to $L^q((0,\infty),X)$ as a bounded
    operator from $L^q((0,\infty),X)$ into itself, \eqref{Lp34} and \eqref{Lp35} implies that
    $$\Big\| \left( - \frac{d^2}{dx^2} \right)^{i \omega} f\Big\|_{L^q(\R,X)}
        \leq C \left( \left\| f_+\right\|_{L^q((0,\infty),X)} + \left\| f_-\right\|_{L^q((0,\infty),X)} \right)
        \leq C \left\| f \right\|_{L^q(\R,X)}, $$
    for every $f \in L^q(\R) \otimes X$.
\end{proof}

Let $\beta>0$ and $f \in \mathcal{S}_\lambda(0,\infty)$. According to Theorem~\ref{boundedness},
there exists a set $\Omega \subset (0,\infty)$, such that $|(0,\infty) \setminus \Omega|=0$ and for every $x \in \Omega$,
the functions $G_{P,\C}^{\lambda,\beta}(\Delta_\lambda^{i\omega}f)(\cdot,x)$ and $G_{P,\C}^{\lambda,\beta+1}(f)(\cdot,x)$
are in $H$. Let $x \in \Omega$. We denote by $A_1$ and $A_2$ the linear bounded operators from $H$ into $\C$ defined by
$$A_1(h)
    = \int_0^\infty G_{P,\C}^{\lambda,\beta}(\Delta_\lambda^{i\omega}f)(t,x) h(t) \frac{dt}{t}, \quad h \in H, $$
and
$$A_2(h)
    = \int_0^\infty G_{P,\C}^{\lambda,\beta+1}(f)(t,x) h(t) \frac{dt}{t}, \quad h \in H.$$
We also define, for every $h \in H$,
$$T_{\omega,\beta}(h)(t)
    = \frac{1}{t^\beta} \int_0^t (t-s)^{\beta-1} h(t-s) \phi_\omega(s)ds, \quad t \in (0,\infty),$$
where $\phi_\omega(s)=s^{-2 i \omega}/\G(1-2 i \omega)$, $s \in (0,\infty)$. Thus, $T_{\omega,\beta}$ is a linear bounded
operator from $H$ into itself. Indeed, Jensen's inequality leads to
\begin{align*}
    \|T_{\omega,\beta}(h)\|_H
        \leq & \left( \int_0^\infty \frac{1}{t^{2\beta+1}} \left( \int_0^t \left| h(t-s)(t-s)^{\beta-1} \phi_\omega(s) \right| ds \right)^2 dt \right)^{1/2} \\
        \leq & C \left( \int_0^\infty \frac{1}{t} \left( \int_0^t \left| h(u)  \right|  \frac{u^{\beta-1}du}{t^\beta}\right)^2 dt \right)^{1/2}
        \leq  C \left( \int_0^\infty \frac{1}{t^{\beta+1}} \int_0^t \left| h(u)  \right|^2 u^{\beta-1} du  dt \right)^{1/2} \\
        \leq & C \|h\|_H, \quad h \in H.
\end{align*}
We now show that
\begin{equation}\label{Lp36}
    A_1(h)
        = - A_2(T_{\omega,\beta} h), \quad h \in H.
\end{equation}
Indeed, let $h \in H$. Since $T_{\omega,\beta} h \in H$, we can write
\begin{align*}
    A_2(T_{\omega,\beta}h)
        & = \int_0^\infty G_{P,\C}^{\lambda,\beta+1}(f)(t,x) (T_{\omega,\beta}h)(t) \frac{dt}{t}
         = \int_0^\infty t^{\beta+1} \partial_t^{\beta+1} P_t^\lambda (f)(x) (T_{\omega,\beta}h)(t) \frac{dt}{t}.
\end{align*}
By Lemma~\ref{Lem4} we have that
\begin{equation}\label{28.2}
    \partial_t^{\beta+1} P_t^\lambda(f)(x)
        = e^{i \pi (\beta+1)} h_\lambda \left( y^{\beta+1} e^{-yt} h_\lambda(f)(y) \right)(x), \quad t,x \in (0,\infty).
\end{equation}
Interchanging the order of integration twice we get
\begin{align*}
    A_2(T_{\omega,\beta}h)
        = & e^{i \pi (\beta+1)} \int_0^\infty t^\beta h_\lambda \left( y^{\beta+1} e^{-yt} h_\lambda(f)(y) \right)(x) (T_{\omega,\beta}h)(t)dt \\
        = & e^{i \pi (\beta+1)}  h_\lambda \left( h_\lambda(f)(y) y^{\beta+1}\int_0^\infty e^{-yt} t^\beta (T_{\omega,\beta}h)(t) dt   \right)(x)\\
        = & e^{i \pi (\beta+1)}  h_\lambda \left( h_\lambda(f)(y)y^{\beta+1}  \int_0^\infty e^{-yt} \int_0^t (t-s)^{\beta-1} h(t-s) \phi_\omega(s) ds dt  \right)(x)\\
        = & e^{i \pi (\beta+1)}  h_\lambda \left( y^{\beta+2 i \omega} h_\lambda(f)(y) \int_0^\infty e^{-yu} u^{\beta-1} h(u) du   \right)(x)\\
        = & e^{i \pi (\beta+1)} \int_0^\infty h_\lambda \left[ y^{\beta+2 i \omega} h_\lambda(f)(y)  e^{-yu} u^{\beta}\right](x) h(u)     \frac{du}{u}\\
        = & - \int_0^\infty h_\lambda \left[ u^{\beta} e^{i\pi \beta}y^{\beta+2 i \omega} e^{-yu} h_\lambda(f)(y)   \right](x) h(u)     \frac{du}{u}\\
        = & - \int_0^\infty u^\beta \partial_u^\beta P_u^\lambda \left[ h_\lambda( y^{2i\omega}h_\lambda(f)(y))   \right](x) h(u)     \frac{du}{u}\\
        = & -A_1(h), \quad h \in H,
\end{align*}
and \eqref{Lp36} is established. Note that the interchanges in the order of integration are justified because the
function $\sqrt{z} J_{\lambda-1/2}(z)$ is bounded on $(0,\infty)$ and $h_\lambda(f) \in \mathcal{S}_\lambda(0,\infty)$.

From \eqref{Lp36} we deduce that, for every $f \in \mathcal{S}_\lambda(0,\infty) \otimes \B$,
\begin{equation*}\label{Lp36.5}
    G_{P,\B}^{\lambda,\beta}(\Delta_\lambda^{i\omega}f)(\cdot,x)
        = - G_{P,\B}^{\lambda,\beta+1}(f)(\cdot,x) \circ T_{\omega, \beta}, \quad \text{a.e.   } x \in (0,\infty),
\end{equation*}
as elements of $L(H,\B)$, the space of linear bounded operators from $H$ into $\B$.

Let $f \in \mathcal{S}_\lambda(0,\infty) \otimes \B$. Since $\Delta_\lambda^{i\omega}f \in L^p(0,\infty) \otimes \mathbb{B}$, \eqref{A1} implies that
\begin{equation}\label{28.1}
    G_{P,\B}^{\lambda,\beta}(\Delta_\lambda^{i\omega}f)(\cdot,x)
        = - G_{P,\B}^{\lambda,\beta+1}(f)(\cdot,x) \circ T_{\omega, \beta}, \quad \text{a.e.   } x \in (0,\infty),
\end{equation}
as elements of $\gamma(H,\B)$. Moreover, according to the ideal property for $\gamma$-radonifying operators \cite[Theorem 6.2]{Nee}, we get
$$\Big\| G_{P,\B}^{\lambda,\beta+1}(f)(\cdot,x) \circ T_{\omega, \beta} \Big\|_{\gamma(H,\B)}
    \leq \| T_{\omega, \beta} \|_{L(H,H)} \Big\| G_{P,\B}^{\lambda,\beta+1}(f)(\cdot,x)  \Big\|_{\gamma(H,\B)}, \quad \text{a.e.  } x \in (0,\infty).$$
Then, \eqref{A1} and \eqref{28.1}  lead to
\begin{align*}
    \Big\| \Delta_\lambda^{i \omega}(f) \Big\|_{L^p((0,\infty),\B)}
        & \leq C \Big\| G_{P,\B}^{\lambda,\beta}(\Delta_\lambda^{i\omega}f) \Big\|_{L^p((0,\infty),\gamma(H,\B))}
         = C \Big\| G_{P,\B}^{\lambda,\beta+1}(f)\circ T_{\omega, \beta} \Big\|_{L^p((0,\infty),\gamma(H,\B))} \\
        & \leq C \Big\| G_{P,\B}^{\lambda,\beta+1}(f)\Big\|_{L^p((0,\infty),\gamma(H,\B))}
         \leq C \|f\|_{L^p((0,\infty),\B)}.
\end{align*}
Hence, $\Delta_\lambda^{i\omega}$ can be extended from $L^p(0,\infty) \otimes \B$ to $L^p((0,\infty),\B)$ as a bounded operator
from $L^p((0,\infty),\B)$ into itself. By Proposition~\ref{ImagBess} we conclude that $\B$ is UMD, and
the proof of Theorem~\ref{Th1.4} is complete.

\section{Proof of Theorem~\ref{Th1.5}}\label{sec:Proof5}

The Bessel operator $\Delta_\lambda$ is positive in $L^2(0,\infty)$. Then, the square root $\sqrt{\Delta}_\lambda$ of
$\Delta_\lambda$ is defined by
$$\sqrt{\Delta}_\lambda f
    = h_\lambda \left( y h_\lambda(f)\right), \quad f \in D(\sqrt{\Delta}_\lambda),$$
where, since $h_\lambda$ is an isometry in $L^2(0,\infty)$, the domain $D(\sqrt{\Delta}_\lambda)$ of $\sqrt{\Delta}_\lambda$
is the following set
$$D(\sqrt{\Delta}_\lambda)
    = \{ f \in L^2(0,\infty) : y h_\lambda(f) \in L^2(0,\infty)\}.$$

The Poisson semigroup $\{P_t^\lambda\}_{t>0}$ is the one generated by the operator $-\sqrt{\Delta}_\lambda$. We
define $M(y)=m(y^2)$, $y \in (0,\infty)$. It is clear that the $\sqrt{\Delta}_\lambda$-multiplier
associated with $M$ coincides with the $\Delta_\lambda$-multiplier defined by $m$. Since the function $M$ satisfies
the conditions specified in \cite[Theorem 1]{Me1}, from the proof of \cite[Theorem 1]{Me1} we deduce that,
for every $n \in \N$, and $f \in \mathcal{S}_\lambda(0,\infty)$,
\begin{equation}\label{29.1}
    t^{n+1} \partial_t^{n+1} P_t^\lambda \left( M(\sqrt{\Delta}_\lambda)f\right)(x)
        = \frac{1}{2\pi} \int_\R \mathcal{M}_n(t,u) t \partial_t P_{t/2}^\lambda \left( \Delta_\lambda^{iu/2} f \right)(x)du, \quad t,x \in (0,\infty),
\end{equation}
where
$$\mathcal{M}_n(t,u)
    = \int_0^\infty y^{-iu-1}M_n(t,y) dy, \quad u \in \R \text{ and } t \in (0,\infty),$$
and
$$M_n(t,y)
    = (ty)^n e^{-ty/2} M(y), \quad t,y \in (0,\infty).$$

We also have that, for every $n \in \N$ and $f \in \mathcal{S}_\lambda(0,\infty) \otimes \B$,
$$t^{n+1} \partial_t^{n+1} P_t^\lambda \left( M(\sqrt{\Delta}_\lambda)f\right)(x)
    = \frac{1}{2\pi} \int_\R \mathcal{M}_n(t,u) t \partial_t P_{t/2}^\lambda \left( \Delta_\lambda^{iu/2} f \right)(x)du, \quad t,x \in (0,\infty).$$
Moreover, according to \cite[Theorem 1]{Me1}, $M(\sqrt{\Delta}_\lambda)f \in L^p(0,\infty)\otimes \B$, $f \in \mathcal{S}_\lambda(0,\infty) \otimes \B$.

Let $n \in \N$. We define, for every $u \in \R$, the operator
$$L_{n,u}(h)(t)
    = \mathcal{M}_n(t,u)h(t),\quad t \in (0,\infty).$$
Since
$$\sup_{\substack{ u \in \R \\ t \in (0,\infty) }} \left| \mathcal{M}_n(t,u) \right|
    \leq C \|m\|_{L^\infty(0,\infty)},$$
the family of operators $\{L_{n,u}\}_{u \in \R}$ is bounded in $L(H,H)$.

Let $f \in \mathcal{S}_\lambda(0,\infty) \otimes \B$. Since $h_\lambda$ is an isometry in $L^2(0,\infty)$, \eqref{23.1} and \eqref{28.2}
allow us to write
$$t \partial_t P_{t/2}^\lambda \left( \Delta_\lambda^{iu/2} f\right)(x)
    = -\frac{1}{2} h_\lambda \left( ty e^{-ty/2} y^{iu} h_\lambda(f)(y) \right)(x), \quad t,x \in (0,\infty) \text{ and } u \in \R.$$
Then, Minkowski's inequality leads to
\begin{align*}
    & \left( \int_0^\infty \left\| t \partial_t P_{t/2}^\lambda \left( \Delta_\lambda^{iu/2}f \right)(x) \right\|_\B^2 \frac{dt}{t} \right)^{1/2}
        \leq C \int_0^\infty \|h_\lambda(f)(y)\|_\B \left( \int_0^\infty \left| t y e^{-ty/2} \right|^2 \frac{dt}{t} \right)^{1/2} dy \\
    & \qquad \leq C \int_0^\infty \|h_\lambda(f)(y)\|_\B dy < \infty, \quad x \in (0,\infty) \text{ and } u \in \R,
\end{align*}
because $h_\lambda(f) \in \mathcal{S}_\lambda(0,\infty) \otimes \B$ and the function $\sqrt{z} J_\nu(z)$ is bounded on $(0,\infty)$
when $\nu>-1/2$. We conclude that
$$t \partial_t P_{t/2}^\lambda \left( \Delta_\lambda^{iu/2} f \right)(x) \in \gamma(H,\B), \quad u \in \R \text{ and } x \in (0,\infty).$$
According to \cite[p. 642]{Me1}, we get
$$\int_\R |\mathcal{M}_n(t,u)| \left\|t \partial_t P_{t/2}^\lambda \left( \Delta_\lambda^{iu/2} f \right)(x) \right\|_\B du
    \in L^p\left( (0,\infty), L^2\left( (0,\infty), dt/t \right)\right),$$
and we infer that
$$\int_0^\infty \left( \int_\R |\mathcal{M}_n(t,u)| \left\|t \partial_t P_{t/2}^\lambda \left( \Delta_\lambda^{iu/2} f \right)(x) \right\|_\B du\right)^2 \frac{dt}{t}
    <\infty, \quad \text{a.e.  } x \in (0,\infty).$$
If $h \in H$ we have that
\begin{align*}
    & \int_0^\infty \int_\R \mathcal{M}_n(t,u) t \partial_t P_{t/2}^\lambda \left( \Delta_\lambda^{iu/2} f \right)(x)   h(t) \frac{dudt}{t} \\
    & \qquad = \int_\R \int_0^\infty \mathcal{M}_n(t,u) t \partial_t P_{t/2}^\lambda \left( \Delta_\lambda^{iu/2} f \right)(x) h(t) \frac{dtdu}{t} ,
    \quad \text{a.e.  } x \in (0,\infty).
\end{align*}
Hence, if $\{h_j\}_{j=1}^k$ is an orthonormal system in $H$ we can write
\begin{align*}
    & \left( \E \Big\| \sum_{j=1}^k \gamma_j
                    \int_0^\infty \int_\R \mathcal{M}_n(t,u) t \partial_t P_{t/2}^\lambda \left( \Delta_\lambda^{iu/2} f \right)(x)   h_j(t) \frac{dudt}{t} \Big\|_\B^2\right)^{1/2} \\
    & \qquad = \left( \E \Big\| \int_\R \sum_{j=1}^k \gamma_j
                    \int_0^\infty  \mathcal{M}_n(t,u) t \partial_t P_{t/2}^\lambda \left( \Delta_\lambda^{iu/2} f \right)(x) h_j(t) \frac{dtdu}{t}   \Big\|_\B^2\right)^{1/2} \\
    & \qquad \leq \int_\R \left( \E \Big\|  \sum_{j=1}^k \gamma_j
                    \int_0^\infty  \mathcal{M}_n(t,u) t \partial_t P_{t/2}^\lambda \left( \Delta_\lambda^{iu/2} f \right)(x) h_j(t) \frac{dt}{t}   \Big\|_\B^2\right)^{1/2} du \\
    & \qquad \leq \int_\R \Big\| \mathcal{M}_n(t,u) t \partial_t P_{t/2}^\lambda \left( \Delta_\lambda^{iu/2} f \right)(x)\Big\|_{\gamma(H,\B)} du,
    \quad \text{a.e.  } x \in (0,\infty).
\end{align*}
Here $\{\gamma_j\}_{j=1}^\infty$ is a sequence of independent Gaussian variables.

We conclude that, for a.e.  $x \in (0,\infty)$,
\begin{equation}\label{30.1}
    \left\| \int_\R  \mathcal{M}_n(t,u) t \partial_t P_{t/2}^\lambda \left( \Delta_\lambda^{iu/2} f \right)(x) du \right\|_{\gamma(H,\B)}
        \leq C   \int_\R  \left\| \mathcal{M}_n(t,u) t \partial_t P_{t/2}^\lambda \left( \Delta_\lambda^{iu/2} f \right)(x)  \right\|_{\gamma(H,\B)} du.
\end{equation}

For every $u \in \R$ we have that
$$\mathcal{M}_n(t,u) t \partial_t P_{t/2}^\lambda \left( \Delta_\lambda^{iu/2} f \right)(x)
    = t \partial_t P_{t/2}^\lambda \left( \Delta_\lambda^{iu/2} f \right)(x) \circ L_{n,u}, \quad \text{a.e.  } x \in (0,\infty),$$
in the sense of equality in $L(H,\B)$. According to \cite[Theorem 6.2]{Nee} we deduce
\begin{align}\label{30.2}
    & \left\| \mathcal{M}_n(t,u) t \partial_t P_{t/2}^\lambda \left( \Delta_\lambda^{iu/2} f \right)(x) \right\|_{\gamma(H,\B)}
        \leq \|L_{n,u}\|_{L(H,H)} \left\| t \partial_t P_{t/2}^\lambda \left( \Delta_\lambda^{iu/2} f \right)(x) \right\|_{\gamma(H,\B)} \nonumber \\
    & \qquad \leq C\sup_{t>0} |\mathcal{M}_n(t,u)| \left\| t \partial_t P_{t/2}^\lambda \left( \Delta_\lambda^{iu/2} f \right)(x) \right\|_{\gamma(H,\B)},
    \quad \text{a.e.  } x \in (0,\infty).
\end{align}

Putting together \eqref{29.1}, \eqref{30.1}, \eqref{30.2} and by taking into account Theorem~\ref{boundedness} and Proposition~\ref{ImagBess} we obtain
\begin{align*}
    \|m(\Delta_\lambda)f\|_{L^p((0,\infty),\B)}
        = & \|M(\sqrt{\Delta_\lambda})f\|_{L^p((0,\infty),\B)}
        \leq  C \left\| G_{P,\B}^{\lambda, n+1} (M(\sqrt{\Delta_\lambda})f)\right\|_{L^p((0,\infty),\gamma(H,\B))} \\
        \leq & C \left\| \int_\R  \mathcal{M}_n(t,u) t \partial_t P_{t/2}^\lambda \left( \Delta_\lambda^{iu/2} f \right)(x) du \right\|_{L^p((0,\infty),\gamma(H,\B))} \\
        \leq & C \int_\R \sup_{t>0} |\mathcal{M}_n(t,u)|   \left\|   G_{P,\B}^{\lambda,1} \left( \Delta_\lambda^{iu/2} f \right)  \right\|_{L^p((0,\infty),\gamma(H,\B))} du \\
        \leq & C \int_\R \sup_{t>0} |\mathcal{M}_n(t,u)|   \left\|   \Delta_\lambda^{iu/2} f   \right\|_{L^p((0,\infty),\B)} du \\
        \leq & C \left(\int_\R \sup_{t>0} |\mathcal{M}_n(t,u)|   \left\|   \Delta_\lambda^{iu/2} \right\|_{L^p((0,\infty),\B) \to L^p((0,\infty),\B) } du \right)
                \|f\|_{L^p((0,\infty),\B)}.
\end{align*}
Hence, $m(\Delta_\lambda)$ can be extended from $\mathcal{S}_\lambda(0,\infty) \otimes \B$ to $L^p((0,\infty),\B)$ as a bounded operator
from $L^p((0,\infty),\B)$ into itself.

\section{Proof of Theorem~\ref{Th1.6}}\label{sec:Proof6}

In order to apply Theorem~\ref{Th1.5} it is necessary to know nice estimations for the norm
$$ \|   \Delta_\lambda^{i\omega} \|_{L^p((0,\infty),\B) \to L^p((0,\infty),\B) }, \quad \omega \in \R \setminus \{0\}.$$

\begin{Prop}\label{Propnorm}
    Let $X$ be a UMD Banach space, $\lambda >0$ and $1<p<\infty$. Then, there exists $C>0$ such that
    $$\| \Delta_\lambda^{i\omega} \|_{L^p((0,\infty),X) \to L^p((0,\infty),X) }
        \leq C e^{\pi |\omega|}, \quad \omega \in \R.$$
    Moreover, if $\lambda \geq 1$ for every $\omega \in \R \setminus \{0\}$,
    $\Delta_\lambda^{i\omega}$ can be extended to $L^1((0,\infty),X)$
    as a bounded operator from $L^1((0,\infty),X)$ into $L^{1,\infty}((0,\infty),X)$, and
    $$\| \Delta_\lambda^{i\omega} \|_{L^1((0,\infty),X) \to L^{1,\infty}((0,\infty),X)}
        \leq C e^{\pi |\omega|},$$
    where $C>0$ does not depend on $\omega$.
\end{Prop}

\begin{proof}
    Let $\omega \in \R \setminus \{0\}$. According to Proposition~\ref{ImagBess} the operator
    $\Delta_\lambda^{i\omega}$ can be extended to $L^p((0,\infty),X)$ as a bounded operator from  $L^p((0,\infty),X)$
    into itself. Moreover, by \eqref{28.1}, for every $f \in \mathcal{S}_\lambda(0,\infty) \otimes X$, we have that
    $$G_{P,X}^{\lambda,1}(\Delta_\lambda^{i\omega}f)(\cdot,x)
        = - G_{P,X}^{\lambda,2}(f)(\cdot,x) \circ T_{\omega}, \quad \text{a.e.  } x \in (0,\infty),$$
    as elements of $\gamma(H,X)$, where
    $$T_\omega(h)(t)
        = \frac{1}{t} \int_0^t h(t-s) \frac{s^{-2 i \omega}}{\G(1-2 i \omega)}ds, \quad h \in H.$$
    As in the proof of Theorem~\ref{Th1.4} we can see that
    $$\|T_\omega\|_{L(H,H)}
        \leq \frac{1}{|\G(1-2 i \omega)|}
        \leq  e^{\pi |\omega|},$$
    and
    \begin{align*}
        \left\| \Delta_\lambda^{i \omega} f \right\|_{L^p((0,\infty),X)}
            \leq & C \left\| G_{P,X}^{\lambda,1}(\Delta_\lambda^{i \omega} f) \right\|_{L^p((0,\infty),\gamma(H,X))} \\
            \leq & C e^{\pi |\omega|}\left\| G_{P,X}^{\lambda,2}( f) \right\|_{L^p((0,\infty),\gamma(H,X))}
            \leq   C e^{\pi |\omega|} \left\| f \right\|_{L^p((0,\infty),X)}, \quad f \in \mathcal{S}_\lambda(0,\infty) \otimes X,
    \end{align*}
    that is,
    \begin{equation}\label{Lp36.9}
        \| \Delta_\lambda^{i\omega} \|_{L^p((0,\infty),X) \to L^p((0,\infty),X) }
            \leq C e^{\pi |\omega|},
    \end{equation}
    where $C>0$ does not depend on $\omega$.

    We are going to show that $\Delta_\lambda^{i\omega}$ is a $X$-valued Calder\'on-Zygmund operator. According to \eqref{Lp30}
    and \eqref{Lp31} we have that
    $$\left| \partial_t W_t^\lambda(x,y) \right|
        \leq C \frac{e^{-c(x-y)^2/t}}{t^{3/2}}, \quad t,x,y \in (0,\infty).$$
    Then,
    \begin{equation}\label{Lp37}
        \left| K^\lambda_\omega(x,y) \right|
            \leq C \int_0^\infty \frac{|t^{-i \omega}|}{|\G(1-i \omega)|} \frac{e^{-c(x-y)^2/t}}{t^{3/2}} dt
            \leq C \frac{e^{\pi |\omega|/2}}{|x-y|}, \quad x,y \in (0,\infty), \ x \neq y.
    \end{equation}
    From \eqref{Lp37} we deduce that, for every $f \in \mathcal{S}_\lambda(0,\infty) \otimes X$,
    $$\int_0^\infty |K^\lambda_\omega(x,y)| \ |f(y)|dy < \infty, \quad x \notin \supp(f).$$
    Hence, for each $f \in \mathcal{S}_\lambda(0,\infty) \otimes X$, \eqref{Lp25} implies that
    $$\Delta_\lambda^{i\omega} f(x)
        = \int_0^\infty K^\lambda_\omega(x,y)f(y)dy, \quad \text{a.e.  } x \notin \supp(f).$$

    We can write
    \begin{align*}
        \partial_x \partial_t W_t^\lambda(x,y)
            = & \partial_x \partial_t \left[ \W_t(x-y) \sqrt{2\pi} \left( \frac{xy}{2t} \right)^{1/2} I_{\lambda-1/2}\left( \frac{xy}{2t}\right)e^{-xy/2t}\right] \\
            = & \partial_x \partial_t [\W_t(x-y) ]\sqrt{2\pi} \left( \frac{xy}{2t} \right)^{1/2} I_{\lambda-1/2}\left( \frac{xy}{2t}\right)e^{-xy/2t} \\
              & + \partial_x  [\W_t(x-y) ]\sqrt{2\pi} \partial_t \left[\left( \frac{xy}{2t} \right)^{1/2} I_{\lambda-1/2}\left( \frac{xy}{2t}\right)e^{-xy/2t} \right] \\
              & + \partial_t  [\W_t(x-y)] \sqrt{2\pi} \partial_x \left[\left( \frac{xy}{2t} \right)^{1/2} I_{\lambda-1/2}\left( \frac{xy}{2t}\right)e^{-xy/2t} \right] \\
              & +  \W_t(x-y) \sqrt{2\pi} \partial_x \partial_t \left[ \left( \frac{xy}{2t} \right)^{1/2} I_{\lambda-1/2}\left( \frac{xy}{2t}\right)e^{-xy/2t}\right] \\
            = & \sum_{j=1}^4 \EE_j(t,x,y), \quad t,x,y \in (0,\infty).
    \end{align*}
    Applying \eqref{Lp29} and \eqref{24.1}  we obtain
    \begin{align*}
        & A)  \qquad \partial_t \left[\left( \frac{xy}{2t} \right)^{1/2} I_\nu\left( \frac{xy}{2t}\right)e^{-xy/2t} \right]
            = - \frac{xy}{2t^2} \frac{d}{dz} \left[z^{\nu + 1/2} z^{-\nu} I_\nu\left( z \right)e^{-z} \right]_{|_{z=xy/2t}} \\
        & \qquad \qquad = - \frac{xy}{2t^2} \left[ (\nu+1/2)z^{\nu-1/2}z^{-\nu} I_{\nu}(z)e^{-z}
                                                + z^{\nu+1/2}z^{-\nu} I_{\nu+1}(z)e^{-z}
                                                - z^{\nu+1/2}z^{-\nu} I_{\nu}(z)e^{-z} \right]_{|_{z=xy/2t}} \\
        & \qquad \qquad = - \frac{1}{\sqrt{2 \pi}} \frac{xy}{2t^2} \left[ \frac{\nu+1/2}{z} \left\{ 1 + \mathcal{O}\left( \frac{1}{z}\right) \right\}
                                                + 1 - \frac{[\nu+1,1]}{2z} + \mathcal{O}\left( \frac{1}{z^2}\right)
                                                - 1 + \frac{[\nu,1]}{2z} + \mathcal{O}\left( \frac{1}{z^2}\right) \right]_{|_{z=xy/2t}}\\
        & \qquad \qquad = \frac{xy}{t^2} \mathcal{O}\left(\left( \frac{t}{xy}\right)^2\right), \quad t,x,y \in (0,\infty);
    \end{align*}
    \begin{align*}
        & B)  \qquad \partial_x \left[\left( \frac{xy}{2t} \right)^{1/2} I_\nu\left( \frac{xy}{2t}\right)e^{-xy/2t} \right]
            = \frac{y}{t} \mathcal{O}\left(\left( \frac{t}{xy}\right)^2\right), \quad t,x,y \in (0,\infty); \hspace{3cm}
    \end{align*}
    \begin{align*}
        & C)  \qquad \partial_x \partial_t \left[\left( \frac{xy}{2t} \right)^{1/2} I_\nu\left( \frac{xy}{2t}\right)e^{-xy/2t} \right]
            = \partial_x \left[ - \frac{xy}{2t^2}\frac{d}{dz}\left[ z^{\nu+1/2}z^{-\nu} I_{\nu}(z)e^{-z} \right]_{|_{z=xy/2t}} \right]\\
        & \qquad \qquad = - \frac{y}{2t^2} \frac{d}{dz}\left[ z^{\nu+1/2}z^{-\nu} I_{\nu}(z)e^{-z} \right]_{|_{z=xy/2t}}
                          - \frac{xy^2}{4t^3} \frac{d^2}{dz^2}\left[ z^{\nu+1/2}z^{-\nu} I_{\nu}(z)e^{-z} \right]_{|_{z=xy/2t}} \\
        & \qquad \qquad = \frac{y}{t^2} \mathcal{O}\left(\left( \frac{t}{xy}\right)^2\right)
                                                - \frac{xy^2}{4t^3} \frac{d}{dz} \Big[ (\nu+1/2)z^{\nu-1/2}z^{-\nu} I_{\nu}(z)e^{-z}
                                                + z^{\nu+1/2}z^{-\nu} I_{\nu+1}(z)e^{-z}\\
        & \qquad \qquad \quad                   - z^{\nu+1/2}z^{-\nu} I_{\nu}(z)e^{-z} \Big]_{|_{z=xy/2t}} \\
        & \qquad \qquad = \frac{y}{t^2} \mathcal{O}\left(\left( \frac{t}{xy}\right)^2\right)
                                - \frac{xy^2}{4t^3}  \Big[\frac{\nu^2-1/4}{z^2}\sqrt{z}I_{\nu}(z)e^{-z}
                                                + \frac{2\nu+2}{z}\sqrt{z}I_{\nu+1}(z)e^{-z} \\
        & \qquad \qquad \quad                  + \sqrt{z} I_{\nu+2}(z)e^{-z}
                                                - 2 \sqrt{z}I_{\nu+1}(z)e^{-z}
                                                + \sqrt{z}I_{\nu}(z)e^{-z}
                                                -\frac{2\nu+1}{z} \sqrt{z} I_{\nu}(z)e^{-z}\Big]_{|_{z=xy/2t}}\\
        & \qquad \qquad = \frac{y}{t^2} \mathcal{O}\left(\left( \frac{t}{xy}\right)^2\right)
                                - \frac{xy^2}{4\sqrt{2\pi}t^3} \Big[ \frac{\nu^2-1/4}{z^2} \left\{ 1 + \mathcal{O}\left( \frac{1}{z}\right)\right\}
                                                            + \frac{2\nu+2}{z} \left\{ 1 - \frac{[\nu+1,1]}{2z}  + \mathcal{O}\left( \frac{1}{z^2}\right)\right\}\\
        & \qquad \qquad \quad                               -\frac{2\nu+1}{z}\left\{ 1 - \frac{[\nu,1]}{2z}  + \mathcal{O}\left( \frac{1}{z^2}\right)\right\}
                                                            + \left\{ 1 - \frac{[\nu+2,1]}{2z} + \frac{[\nu+2,2]}{4z^2}  + \mathcal{O}\left( \frac{1}{z^3}\right)\right\} \\
        & \qquad \qquad \quad                               - 2 \left\{ 1 - \frac{[\nu+1,1]}{2z} + \frac{[\nu+1,2]}{4z^2}  + \mathcal{O}\left( \frac{1}{z^3}\right)\right\}
                                                            + \left\{ 1 - \frac{[\nu,1]}{2z} + \frac{[\nu,2]}{4z^2}  + \mathcal{O}\left( \frac{1}{z^3}\right)\right\}\Big]_{|_{z=xy/2t}} \\
        & \qquad \qquad = \frac{xy^2}{t^3} \mathcal{O}\left( \left( \frac{t}{xy} \right)^3 \right), \quad t,x,y \in (0,\infty).
    \end{align*}
    Here $\nu=\lambda-1/2$.
    Then, we deduce \\

    $\displaystyle \bullet |\EE_1(t,x,y)|
        \leq C \frac{e^{-c(x-y)^2/t}}{t^2}, \quad t,x,y \in (0,\infty)$,\\

    $\displaystyle \bullet |\EE_2(t,x,y)|
        \leq C \frac{e^{-c(x-y)^2/t}}{t} \frac{xy}{t^2} \frac{t^2}{(xy)^2}, \quad t,x,y \in (0,\infty)$,\\

    $\displaystyle \bullet |\EE_3(t,x,y)|
        \leq C \frac{e^{-c(x-y)^2/t}}{t^{3/2}} \frac{y}{t} \frac{t^2}{(xy)^2}, \quad t,x,y \in (0,\infty)$,\\

    \noindent and

    $\displaystyle \bullet |\EE_4(t,x,y)|
        \leq C \frac{e^{-c(x-y)^2/t}}{t^{1/2}} \frac{xy^2}{t^3} \frac{t^3}{(xy)^3}, \quad t,x,y \in (0,\infty)$.\\

    We now estimate
    $$\int_0^{xy/2} |\EE_j(t,x,y)| dt, \quad j=1,2,3,4.$$
    Firstly, we have that
    \begin{align*}
        \int_0^{xy/2} |\EE_1(t,x,y)| dt
            \leq & C \int_0^\infty \frac{e^{-c(x-y)^2/t}}{t^2} dt
            \leq \frac{C}{|x-y|^2}, \quad x,y \in (0,\infty), \ x \neq y,
    \end{align*}
    and also
    \begin{align*}
        \int_0^{xy/2} |\EE_2(t,x,y)| dt
            \leq & C \int_0^{xy/2} \frac{e^{-c(x-y)^2/t}}{txy} dt
            \leq C \int_0^\infty \frac{e^{-c(x-y)^2/t}}{t^2} dt
            \leq \frac{C}{|x-y|^2}, \quad x,y \in (0,\infty), \ x \neq y.
    \end{align*}
    To study $\EE_3$ and $\EE_4$ we distingue two cases
    \begin{align*}
        & \int_0^{xy/2} \left( |\EE_3(t,x,y)| + |\EE_4(t,x,y)| \right) dt
            \leq  C \int_0^{xy/2} \frac{e^{-c(x-y)^2/t}}{\sqrt{t}} \frac{dt}{x^2y} \\
        & \qquad \qquad    \leq \left\{ \begin{array}{l}
                            C \displaystyle \int_0^{xy/2} \frac{e^{-c(x-y)^2/t}}{\sqrt{t}} \frac{1}{x^2y} \left( \frac{xy}{t}\right)^{3/2} dt
                                \leq C \sqrt{\frac{y}{x}}\int_0^{xy/2} \frac{e^{-c(x-y)^2/t}}{t^2}  dt \\
                            C \displaystyle \int_0^{xy/2} \frac{e^{-c(x-y)^2/t}}{\sqrt{t}} \frac{1}{x^2y} \left( \frac{xy}{t}\right)^{2} dt
                                \leq C y\int_0^{xy/2} \frac{e^{-c(x-y)^2/t}}{t^{5/2}}  dt
                         \end{array}\right. \\
        & \qquad \qquad    \leq \left\{ \begin{array}{l}
                            C \displaystyle \int_0^\infty \frac{e^{-c(x-y)^2/t}}{t^2} dt
                                \leq \frac{C}{|x-y|^2}, \quad y<2x, \ x \in (0,\infty), \\
                            C \displaystyle y \int_0^\infty \frac{e^{-c(x-y)^2/t}}{t^{5/2}} dt
                                \leq C \frac{y}{|x-y|^3} \leq \frac{C}{|x-y|^2}, \quad y \geq 2x, \ x \in (0,\infty).
                         \end{array}\right.
    \end{align*}
    Hence, we conclude that
    \begin{equation}\label{Lp38}
        \int_0^{xy/2} \left| \partial_x \partial_t W_t^\lambda(x,y) \right| dt
            \leq \frac{C}{|x-y|^2}, \quad x,y \in (0,\infty), \ x \neq y.
    \end{equation}

    According to \eqref{Lp28} and by taking in mind the above calculations we get
    \begin{align*}
        \left| \partial_x \partial_t W_t^\lambda(x,y) \right|
            \leq & C \frac{e^{-c(x^2+y^2)/t}}{t^2} \left[ \left( \frac{xy}{t}\right)^\lambda + \frac{y}{\sqrt{t}} \left( \frac{xy}{t} \right)^{\lambda-1} \right] \\
            \leq & C \frac{e^{-c(x^2+y^2)/t}}{t^2} \left( 1 + \frac{y}{\sqrt{t}} \right), \quad t,x,y \in (0,\infty), \ xy \leq 2t,
    \end{align*}
    provided that $\lambda \geq 1$. Then,
    \begin{align}\label{Lp39}
        \int_{xy/2}^\infty \left| \partial_x \partial_t W_t^\lambda(x,y) \right| dt
            \leq & C \int_{0}^\infty \frac{e^{-c(x^2+y^2)/t}}{t^2} \left( 1 + \frac{y}{\sqrt{t}} \right) dt \nonumber \\
            \leq & C \left( \frac{1}{x^2+y^2} + \frac{y}{(x^2+y^2)^{3/2}} \right)
            \leq \frac{C}{|x-y|^2} , \quad x,y \in (0,\infty), \ x \neq y.
    \end{align}
    From \eqref{Lp38} and \eqref{Lp39} we deduce that
    \begin{equation}\label{Lp40}
        \left| \partial_x K^\lambda_\omega(x,y)\right|
            \leq C \frac{e^{\pi |\omega|/2}}{|x-y|^2}, \quad x,y \in (0,\infty), \ x \neq y.
    \end{equation}
    Since $K^\lambda_\omega(x,y)=K^\lambda_\omega(y,x)$, $x,y \in (0,\infty)$, we also have that
    \begin{equation}\label{Lp41}
        \left| \partial_y K^\lambda_\omega(x,y)\right|
            \leq C \frac{e^{\pi |\omega|/2}}{|x-y|^2}, \quad x,y \in (0,\infty), \ x \neq y.
    \end{equation}
    By \eqref{Lp37}, \eqref{Lp40} and \eqref{Lp41}, $K^\lambda_\omega$ is a standard Calder\'on-Zygmund kernel.

    By applying now the Calder\'on-Zygmund theory for Banach valued singular integral we obtain that the operator
    $\Delta_\lambda^{i\omega}$ can be extended to $L^1((0,\infty),X)$ as a bounded operator from $L^1((0,\infty),X)$
    into $L^{1,\infty}((0,\infty),X)$. Moreover, \eqref{Lp36.9}, \eqref{Lp37}, \eqref{Lp40} and \eqref{Lp41} lead to
    $$\left\| \Delta_\lambda^{i \omega} \right\|_{L^{1}((0,\infty),X) \to L^{1,\infty}((0,\infty),X)}
        \leq C e^{\pi |\omega|},$$
    where $C>0$ does not depend on $\omega$.
\end{proof}

\begin{Prop}\label{Prop norm1}
    Let $\HH$ be a Hilbert space and $\lambda>0$.
    Then, $\|\Delta_\lambda^{i\omega}\|_{L^2((0,\infty),\HH)}=1$, for every $\omega \in \R \setminus \{0\}$.
\end{Prop}

\begin{proof}
    We consider $f \in L^2(0,\infty)\otimes \HH$, that is, $f=\sum_{j=1}^n a_j f_j$ where $a_j \in \HH$ and $f_j \in L^2(0,\infty)$.
    By using Plancherel equality for Hankel transforms on $L^2(0,\infty)$ we can write
    \begin{align*}
        \int_0^\infty \left\| h_\lambda(f)(x)\right\|_\HH^2 dx
            = & \int_0^\infty \langle h_\lambda(f)(x) , h_\lambda(f)(x) \rangle_\HH dx \\
            = & \sum_{i,j=1}^n \langle a_i , a_j \rangle_\HH \int_0^\infty h_\lambda(f_i)(x) h_\lambda(f_j)(x) dx
            = \sum_{i,j=1}^n \langle a_i , a_j \rangle_\HH \int_0^\infty f_i(x) f_j(x) dx \\
            = & \int_0^\infty \langle f(x) , f(x) \rangle_\HH dx
            = \int_0^\infty \|f(x)\|_\HH^2 dx.
    \end{align*}
    Hence, $h_\lambda$ can be extended to $L^2((0,\infty),\HH)$ as a bounded operator from $L^2((0,\infty),\HH)$ into itself.
    Since $|y^{2i \omega}|=1$, $y \in (0,\infty)$ and $\omega \in \R \setminus \{0\}$, by \eqref{23.1} we conclude that, for every
    $\omega \in \R \setminus \{0\}$, $\Delta_\lambda^{i\omega}$ is bounded from $L^2((0,\infty),\HH)$ into itself and
    $$\|\Delta_\lambda^{i\omega}\|_{L^2((0,\infty),\HH) \to L^2((0,\infty),\HH)}=1.$$
\end{proof}

Let $\omega \in \R \setminus \{0\}$ and assume that $\B=[\HH,X]_\theta$, where $\HH$ is a Hilbert space and $X$ is a UMD space,
$0<\theta<\vartheta/\pi$.
Then, by using the interpolation theorem for vector-valued Lebesgue spaces \cite[Theorem 5.1.2]{BL}, Propositions~\ref{Propnorm} and \ref{Prop norm1}
we deduce that, $\Delta_\lambda^{i\omega}$ is a
bounded operator from $L^p((0,\infty),\B)$ into itself, being $p=2/(1+\theta)$ and
\begin{align*}
    \|\Delta_\lambda^{i\omega}\|_{L^p((0,\infty),\B) \to L^p((0,\infty),\B)}
        \leq & C \|\Delta_\lambda^{i\omega}\|_{L^2((0,\infty),\HH) \to L^2((0,\infty),\HH)}^{1-\theta} \|\Delta_\lambda^{i\omega}\|_{L^1((0,\infty),X) \to L^{1,\infty}((0,\infty),X)}^{\theta}\\
        \leq & C e^{2\pi(1/p-1/2)|\omega|}.
\end{align*}
Here $C>0$ does not depend on $\omega$.

Since $\Delta_\lambda^{i\omega}$ is selfadjoint, by using duality and that $[\HH,X]_\theta^*=[\HH^*,X^*]_\theta$ (see \cite[p. 1007]{Hy}) we get
$$\|\Delta_\lambda^{i\omega}\|_{L^{p'}((0,\infty),\B) \to L^{p'}((0,\infty),\B)}
    \leq C e^{2\pi(1/p-1/2)|\omega|}.$$
Hence, another interpolation leads to
\begin{equation}\label{Lp42}
    \|\Delta_\lambda^{i\omega}\|_{L^{q}((0,\infty),\B) \to L^{q}((0,\infty),\B)}
        \leq C e^{2\pi(1/p-1/2)|\omega|}, \quad p \leq q \leq p'.
\end{equation}

Since $m$ is a bounded holomorphic function in $\sum_\vartheta$, the function $M(y)=m(y^2)$, $y \in (0,\infty)$, is bounded and holomorphic in
$\sum_{\vartheta/2}$.
The proof now can be finished by proceeding as in the proof of \cite[Theorem 3]{Me1} and by using \eqref{Lp42}.



\end{document}